\documentclass[12pt]{amsart}

\usepackage{amsmath,amsfonts,latexsym,graphicx,amssymb,,amsthm}
\usepackage[tips,matrix,arrow,frame]{xy}
\newcommand{\C}{\mathbb{C}}
\newcommand{\N}{\mathbb{N}}

\newcommand{\Z}{\mathbb{Z}}
\newcommand{\K}{\mathbb{K}}

\newcommand{\fsl}{\mathfrak{sl}}

\newcommand{\ad}{\mathrm{ad}}
\newcommand{\ext}{\mathrm{ext}}
\newcommand{\End}{\mathrm{End}}
\newcommand{\Ker}{\mathrm{Ker}\,}

\newcommand\GCD{\text{GCD}}

\newcommand{\id}{\mathrm{id}}

\newcommand{\pr}{\mathrm{pr}}

\newcommand{\DD}{\mathcal{D}}
\newcommand{\Span}{\text{Span}}
\newcommand{\Rad}{\text{Rad}}
\newcommand{\Der}{\text{Der}}
\newcommand{\lcP}{\overline{\cP}}

\newtheorem{thm}{Theorem}[section]
\newtheorem{prop}[thm]{Proposition}
\newtheorem{lemma}[thm]{Lemma}
\newtheorem{cor}[thm]{Corollary}
\newtheorem{defn}{Definition}[section]

\newtheorem{rem}[thm]{Remark}

\newcommand{\fg}{\mathfrak{g}}

\newcommand{\Hom}{{\mathrm {Hom}}}

\newcommand{\vep}{\varepsilon}

\newcommand{\bW}{\mathbf{W}}
\newcommand{\cA}{\mathcal{A}}

\newcommand{\cC}{\mathcal{C}}
\newcommand{\cD}{\mathcal{D}}

\newcommand{\cL}{\mathcal{L}}
\newcommand{\cM}{\mathcal{M}}

\newcommand{\cP}{\mathcal{P}}
\newcommand{\cS}{\mathcal{S}}

\newcommand{\cV}{\mathcal{V}}

\begin{document}
\title[Simple cuspidal modules over a lattice Lie algebra]{Classification of simple cuspidal modules \\ over a lattice Lie algebra of Witt type}
\author{Y. Billig and K. Iohara}
\address{School of Mathematics and Statistics, Carleton University, Ottawa, Canada}
\email{billig@math.carleton.ca}
\address{Universit\'{e} Lyon, Universit\'{e} Claude Bernard Lyon 1, CNRS UMR 5208, Institut Camille Jordan, 
43 Boulevard du 11 Novembre 1918, F-69622 Villeurbanne cedex, France}
\email{iohara@math.univ-lyon1.fr}

\begin{abstract} Let $W_\pi$ be the lattice Lie algebra of Witt type associated with an additive inclusion $\pi: \Z^N \hookrightarrow \C^2$ with $N>1$. In this article, the classification of simple $\Z^N$-graded $W_\pi$-modules, whose multiplicities are uniformly bounded, is given. 
\end{abstract}

\maketitle

 {}\let\thefootnote\relax

\tableofcontents

\section*{Introduction}
In 1939, H.~Zassenhaus published a long influential article \cite{Z} where he launched basic
theory of Lie algebras over a field of positive characteristic. In particular, he showed that the Lie algebra $\Der(\K[\Z/p\Z])=(\K[T]/(T^p))\dfrac{d}{dT}$, where $\K$ is a field of characteristic $p>2$, is simple and its Killing form is trivial. In 1966, a central extension of this Lie algebra was found by R.~Block \cite{Bl}. Its characteristic $0$ version (without central extension), called the Witt algebra $\bW=\C[T^{\pm 1}]\dfrac{d}{dT}$, had been considered a long time ago by \'E.~Cartan, see, e.g., \cite{C}. 
Since its appearance in quantum field theory in the late 1960's, this Lie algebra had been studied extensively both in mathematics and in physics. 

One of the important results obtained for this Lie algebra is a classification of its 
Harish-Chandra modules. In 1992,
O.~Mathieu \cite{M} showed that a simple $\Z$-graded module over the Virasoro algebra, the central extension of the Witt algebra $\bW$,  whose weight multiplicities are finite, is either a highest/lowest weight module or a module of intermediate series, that is, a $\Z$-graded module whose weight multiplicities are bounded by $1$. The intermediate series for the Witt algebra $\bW$ was previously classified by I.~Kaplansky and 
L.~J.~Santharoubane \cite{KS}. 

The Witt algebra $\bW$ has a basis formed by $L_m=T^{m+1}\dfrac{d}{dT}$ ($m \in \Z$) and these elements satisfy the commutation relations:
\[   [L_m, L_n]=(n-m)L_{m+n}. \]
This Lie algebra has several generalizations; for example,
 the Lie algebra of algebraic vector fields on a torus $(\C^\ast)^N$ for some $N > 1$ (cf. \cite{BF2}). For any lattice $\Lambda$ of finite rank and an additive injective map $l: \Lambda \rightarrow \C$, H. Zassenhaus \cite{Z} (see pages 47-48)
considered the Lie algebra $W_l$ with basis $(L_\lambda)_{\lambda \in \Lambda}$ satisfying
\[ [L_\lambda, L_\mu]=l(\mu-\lambda)L_{\lambda+\mu}. \]
Such a Lie algebra is called the generalized Witt algebra or solenoidal Lie algebra (cf. \cite{BF2}).  

Y. Billig and V. Futorny classified in \cite{BF1} simple $\Z^N$-graded modules for the Lie algebra of vector fields on an $N$-dimensional torus.
Just as in case of the Witt algebra $\bW$, there are two classes of modules: generalized highest weight modules and modules with bounded multiplicities.
Let us briefly explain a new method developed in \cite{BF1} for the classification of the modules with bounded multiplicities, also known as the cuspidal 
modules. Let $\cV$ be either a Lie algebra of vector fields on a torus or a solenoidal Lie algebra.
A key feature of these Lie algebras is that they are free modules over the ring $\cA=\C[t_1^{\pm 1}, \cdots, t_N^{\pm 1}]$. Based on this fact, one first classifies irreducible $\cA\cV$-modules, that is $\cV$-modules on which one can define a compatible $\cA$-action. 
For a $\cV$-module $M$, one defines its cover $\widehat{M}$ in the category of $\cA\cV$-modules as a certain submodule of the coinduced module
$\Hom_\C (\cA, M)$. The key step is to prove that a cover of a cuspidal module is again cuspidal.
Finally it is shown that any simple $\Z^N$-graded $\cV$-module with uniformly bounded multiplicities can be realized as a quotient of the restriction to $\cV$ of an irreducible cuspidal $\cA\cV$-module. 

In 2013, K. Iohara and O. Mathieu \cite{IM1} classified lattice Lie algebras, that is $\Z^N$-graded simple Lie algebras whose multiplicities are $1$. Among such Lie algebras, there is a class defined below, which is another generalization of the Witt algebra.

Let $\langle \cdot, \cdot \rangle$ be the standard symplectic form on $\C^2$. The Poisson algebra $\cS$ of the principal symbols of pseudo-differential operators in one variable is the vector space $\cS=\bigoplus_{\lambda \in \C^2} \C L_\lambda$ equipped with the commutative algebra structure 
$$L_\lambda \cdot L_\mu=L_{\lambda+\mu+\rho},$$ 
where $\rho=(1,1)$ is a fixed vector, with the Poisson bracket $\{ \cdot, \cdot \}$ given by
\[ \{ L_\lambda, L_\mu\}=\langle \lambda+\rho, \mu+\rho \rangle L_{\lambda+\mu}. \]
For any additive injective map $\pi: \Lambda \hookrightarrow \C^2$, the Lie algebra $W_\pi$ is the Lie subalgebra of $\cS$ generated by $(L_\lambda)_{\lambda \in \pi(\Lambda)}$.The intermediate series for $W_\pi$ has been classified by the same authors \cite{IM2}.

Throughout this paper we assume the following technical condition:
$$(\cC) \hbox{\hskip 1cm}
\langle \alpha+2\rho, \beta \rangle \neq 0 \text{ \ for any \ } \alpha \in \pi(\Lambda) \text{ \ and \ } \beta \in \pi(\Lambda) \setminus \{0\}. 
\hbox{\hskip 1cm}
$$

Since $W_\pi$ has $\cA_\pi$-module structure, where $\cA_\pi=\bigoplus_{\lambda \in \pi(\Lambda)} \C L_{\lambda-\rho}$,  and the commutative ring $\cA_\pi$ has $W_\pi$-module structure,  we can apply the machinery of $\cA\cV$-modules with $\cA = \cA_\pi$ and $\cV = W_\pi$.

In this article, we classify simple $\Lambda$-graded $W_\pi$-modules $M=\bigoplus_{\alpha \in \Lambda} M_\alpha$ with a uniform bound on $\dim M_\alpha$, following the approach developed in \cite{BF1}. 
Our main result states that such a simple $W_\pi$-module is either 
a quotient of
$\cS_{\Gamma}:=\bigoplus_{\mu \in \Gamma} \C L_\mu$, where 
$\Gamma = \beta + \pi(\Lambda)$ for some $\beta \in \C^2$ or of the form $\cS_{\Gamma} \otimes \C^n$ with $n \geq 3$ on which a non-trivial action of $W_\pi$ is defined in section \ref{sect_mult+1}.


This article is organized as follows. In Section 1, we define our basic objects, the Lattice Lie algebra $W_\pi$ of Witt type, some examples of graded $W_\pi$-modules, and state the main result of this article. In Section 2, we discuss a connection between  $W_\pi$-modules and $\cA\cV_\pi$-modules.
We classify simple cuspidal $\cA\cV_\pi$-modules in Section 3. In Section 4,  we prove our main theorem classifying simple cuspidal $W_\pi$-modules.

\noindent \textbf{Acknowledgement.}\,
\textit{
The first author gratefully acknowledges support from the Natural Sciences and Engineering Research Council of Canada. The second author is partially supported by the French ANR (ANR project ANR-15-CE40-0012). He would like to also thank the Embassy of France in Canada for the mobility grant that allowed him to visit
Carleton University. This work was partially supported by the LABEX MILYON (ANR-10-LABX-0070) of Universit\'{e} de Lyon. 
We would also like to thank the University of Lyon and Carleton University for the hospitality 
during our respective visits.}

\section{Lie algebras of Witt type}
In this section, we introduce the Poisson algebra $\cS$ of the principal symbols of pseudo-differential operators on $\C$. As a Lie subalgebra, we introduce a lattice Lie algebra of Witt type $W_\pi$ depending on the additive embedding $\pi: \Lambda \hookrightarrow \C^2$, where $\Lambda=\Z^N$ with $N>1$.   

\subsection{Poisson algebra $\cS$}\label{sect_Poisson-P}

The Poisson algebra $\cS$ of the symbols of twisted pseudo-differential operators  on a circle is the  algebra  with a basis $(L_\lambda)_{\lambda \in \C^2}$ with products given by
\begin{align*}
L_\lambda \cdot L_\mu=
&L_{\lambda+\mu+\rho}, \\
\{L_\lambda, L_\mu\}=
&\langle \lambda+\rho, \mu+\rho \rangle L_{\lambda+\mu}, 
\end{align*}
where $\langle \cdot, \cdot\rangle$ is the standard symplectic structure of $\C^2$ and  $\rho=(1,1)$. Indeed for $\lambda=(a,b)$,
$L_\lambda$ is the symbol of the twisted pseudo-differential operator $z^{a+1} (d/dz)^{b+1}$, see \cite{IM1} for details.

The Poisson algebra $\cS$ viewed as a commutative algebra is also denoted by $\cA$. 

For an element $\xi \in \C^2 \setminus \C \rho$, the subspace $\bigoplus_{m \in \Z} \C L_{m\xi}$ is stable under the Poisson bracket and is isomorphic to the Witt algebra, the $\Z$-graded Lie algebra $\bW$ with basis $(L_m)_{m\in\Z}$ and bracket given by $[L_m,L_n]=(n-m) L_{n+m}$. This generalizes to a $\Lambda$-graded algebra as follows.

Given an injective additive map $\pi:\Lambda \hookrightarrow \C^2$, let $W_\pi \subset \cS$ be the Lie subalgebra with basis $(L_\lambda)_{\lambda \in \pi(\Lambda)}$. In what follows, we will always assume:
\begin{enumerate}
\item[i)] $\mathrm{Im}~\pi \not\subset \C \rho$, 
\item[ii)] $2\rho \not\in \mathrm{Im}~\pi$,
\item[iii)] $\pi(\Lambda)$ does not lie in a complex line.
\end{enumerate}
The condition i) is necessary since otherwise $W_\pi$ is commutative. 
It is easy to show that $W_\pi$ is simple if and only if the condition ii) holds (cf. Lemma 49 of \cite{IM1}).
The condition iii) implies that we exclude the case when $W_\pi$ becomes a so-called 
solenoidal Lie algebra (see, e.g., \cite{BF2}). 
Besides these conditions, we assume
$$(\cC) \hbox{\hskip 1cm}
\langle \alpha+2\rho, \beta \rangle \neq 0 \text{ \ for any \ } \alpha \in \pi(\Lambda) \text{ \ and \ } \beta \in \pi(\Lambda) \setminus \{0\}. 
\hbox{\hskip 1cm}
$$

We will indicate clearly, whenever we need this condition. 
\begin{rem} 
It seems that all results of this paper may be established under a weaker assumption
of the existence of
a $\Z$-basis $\{\vep_k\}_{1\leq k\leq N}$ of $\Lambda$ such that $\langle \alpha+2\rho, \vep_k \rangle\neq 0$ for all $1\leq k\leq N$ and $\alpha \in \pi(\Lambda)$. 
\end{rem}

Let $\cA_\pi:=\bigoplus_{\lambda \in \pi(\Lambda)} \C L_{\lambda-\rho}$ be a commutative subalgebra of $\cA$. It is clear that, since $\pi$ is injective by assumption, $\cA_\pi$ is isomorphic to $\C[t_1^{\pm 1}, \cdots, t_N^{\pm 1}]$.
The following Lemma follows immediately from the definitions:
\begin{lemma}\label{lemma_elementary-P}
\begin{enumerate}
\item The Lie algebra $W_\pi$ has an $\cA_\pi$-module structure.
\item The commutative algebra $\cA_\pi$ has a $W_\pi$-module structure.
\end{enumerate}
\end{lemma}
Notice that $L_{-\rho}$ is the identity element of $\cA_\pi$ and $L_{\lambda-\rho} \cdot L_{-\lambda-\rho}=L_{-\rho}$ for any $\lambda \in \C^2$.

\subsection{Definition of $\cA \cV_\pi$-modules}
An  $\cA \cV_\pi$-module $T$ is a vector space which is simultaneously a module for the unital commutative associative algebra $\cA_\pi$
and a module over the Lie algebra $W_\pi$ with the two actions being compatible:
\begin{equation*}
\theta^\cV(L_\lambda)\theta^{\cA}(L_{\mu-\rho})=\theta^{\cA}(\{L_\lambda, L_{\mu-\rho}\})+\theta^{\cA}(L_{\mu-\rho})\theta^\cV(L_\lambda),
\text{ for } \lambda, \mu \in \pi(\Lambda),
\end{equation*}
where $\theta^\cV: W_\pi \rightarrow \End(T)$ and $\theta^{\cA}: \cA_\pi \rightarrow \End(T)$ denote the module structures.
To simplify the notations, we may also write $L_\lambda^\cV$ in place of $\theta^\cV(L_\lambda)$ and $L_{\lambda-\rho}^{\cA}$ for $\theta^{\cA}(L_{\lambda-\rho})$.

Here is a basic example.  As the Poisson algebra $\cS$ has a $\C^2$-graded structure, for each coset $\Gamma \in \C^2/\pi(\Lambda)$, the subspace of $\cS$ defined by $\cS_\Gamma=\bigoplus_{\mu \in \Gamma} \C L_\mu$ is stable under the action of $W_\pi$. Notice that this $W_\pi$-module also has the structure of an $\cA_\pi$-module. 
The $W_\pi$-module structure of $\cS_\Gamma$ is described by the following:
\begin{lemma}[cf. \cite{IM2}]\label{lemma_irr-tensor}
Let $\Gamma \in \C^2/\pi(\Lambda)$. 
\begin{enumerate}
\item The $W_\pi$-module $\cS_\Gamma$ is irreducible iff $-\rho, -2\rho \not\in \Gamma$.
\item If $-\rho \in \Gamma$, then there is an inclusion of the trivial 1-dimensional $W_\pi$-module $\C L_{-\rho} \hookrightarrow \cS_\Gamma$. 
Set $\overline{M}=\cS_{-\rho + \pi(\Lambda)}/\C L_{-\rho}$.
\item If $-2\rho \in \Gamma$, then there is an inclusion of the $W_\pi$-module $\overline{M}^\vee \hookrightarrow \cS_\Gamma$ and its cokernel is a trivial representation. 
\end{enumerate}
Here, $\overline{M}^\vee$ signifies the restricted dual of $\overline{M}$.
\end{lemma}
Notice that, in general, the restricted dual of the $W_\pi$-module $\cS_{\beta+\pi(\Lambda)}$ is isomorphic to $\cS_{-\beta-3\rho+\pi(\Lambda)}$. By the condition iii) on the additive map $\pi: \Lambda \hookrightarrow \C^2$, the $W_\pi$-modules $\overline{M}$ and $\overline{M}^\vee$ are not isomorphic.
\subsection{$\Lambda$-graded $W_\pi$-modules}
A $W_\pi$-module $M$ is said to be \textbf{$\Lambda$-graded} if there exists a coset $\Gamma \in \C^2/\pi(\Lambda)$ such that 
$$M=\bigoplus_{\mu \in \Gamma} M_\mu,$$ 
where 
$$M_\mu = \left\{ m \in M \, | \, L_0.m = \langle \rho, \mu \rangle m \right\}.$$
It is easy to check that  
$L_\lambda.M_\mu \subset M_{\lambda+\mu}$ for any $\lambda\in \pi(\Lambda)$ and $\mu \in \Gamma$.

Note that the condition $(\cC)$ implies that the additive map $\langle \rho, \cdot \rangle: \Lambda \rightarrow \C, \ \alpha \mapsto \langle \rho, \alpha \rangle$, is injective.

Similarly, an $\cA\cV_\pi$-module is  \textbf{$\Lambda$-graded} if it is $\Lambda$-graded as a $W_\pi$-module. 
For such modules we also have
$\theta^{\cA}(L_{\lambda-\rho})(M_\alpha) \subset M_{\alpha+\lambda}$
for $\lambda \in \pi(\Lambda)$.

%

Suppose $T$ is an irreducible $\Lambda$-graded $\cA \cV_\pi$-module with $T_{\beta} \neq (0)$. Then $T$ is supported on a coset
$\Gamma = \beta + \pi(\Lambda)$:
$$ T=\bigoplus_{\alpha \in \Gamma} T_\alpha .$$

\vskip+0.1in

A $\Lambda$-graded $W_\pi$-module $M$ 
(resp. $\cA\cV_\pi$-module $T$) 
is said to be 
\begin{enumerate}
\item
\textbf{simple graded} if $M$ (resp. $T$) does not contain a proper $\Lambda$-graded submodule,  
\item 
\textbf{cuspidal} if
the dimensions of homogeneous components of $M$ (resp. $T$) are uniformly bounded.
\end{enumerate}
Notice that, by the condition $(\cC)$, any submodule of a graded module is graded.
\subsection{Irreducible $W_\pi$-modules with multiplicities $>1$}\label{sect_mult+1}

It has been proved in \cite{M} and \cite{MP} that all of the homogeneous components
of a simple cuspidal $\bW$-module have dimension $\leq 1$. Here we recall a construction of  cuspidal $W_\pi$-modules with the homogeneous components of 
an arbitrary dimension $d\geq 1$, given in \cite{IM2}.


Let $V$ be a $\C$-vector space and $\langle \cdot \vert \cdot \rangle$ be a skew-symmetric bilinear form on $V$. It induces a Poisson bracket on
the symmetric algebra $S^\bullet V$ satisfying
\[
\{\alpha,\beta\}=\langle \alpha \vert \beta \rangle 1
\]
for $\alpha,\,\beta\in V$. Since $\{S^mV,S^n V\}\subset S^{m+n-2}V$, it follows that $S^2V$ is a Lie subalgebra and each component $S^n V$ is an $S^2V$-module. \\
In the rest of this subsection, we consider the case $V=\C^2$ with the standard symplectic form $\langle \cdot , \cdot \rangle$ on $\C^2$ used to define the Poisson algebra $\cS$ in \S \ref{sect_Poisson-P}. 
In such a case, $S^2V$ is isomorphic to $\fsl_2$ (see section 4.1 below) and $S^nV$ is the irreducible $\fsl_2$-module of dimension $n+1$. 

Since $\cS$ is a Poisson algebra, it will be convenient to denote by $\cV$ the underlying Lie algebra and by $\cA$ the underlying commutative algebra.
As vector spaces, $\cS = \cV = \cA$.
Set
\[ \cS^{\ext}=\cV \ltimes \left( \cA \otimes S^2V \right).
\]
\noindent Clearly  $\cS^{\ext}$ has a structure of Lie algebra,
and  for any $n$, $\cA\otimes S^nV$ is an $\cS^{\ext}$-module.
Define a linear map $c: \cV\rightarrow \cA\otimes S^2V$ by
the formula:
\[ c(L_\lambda^{\cV})=\frac{1}{2}  L_{\lambda-\rho}^{\cA}\otimes \lambda(\lambda+\rho). \]
For $X\in \cV$,  set $j(X)=X+c(X)$. 

\begin{lemma}\label{lemma_MC}

The map $j:\cV \rightarrow \cS^{ext}$ is 
a Lie algebra morphism, i.e., the map $c$ 
satisfies the Maurer-Cartan equation

$$c([X,Y])= X.c(Y)-Y.c(X)+ [c(X),c(Y)]$$

\noindent for any $X,\,Y\in \cV$.

\end{lemma}  

For any $n\geq 0$, $\cA \otimes S^nV$ is naturally an $\cS^{ext}$-module.
Then  $\cM^n:=j^\ast\cA \otimes S^nV$ is a $\C^2$-graded $\cV$-module,
with all homogeneous components of dimension $n+1$. Given a coset
$\Gamma \in \C^2/\pi(\Lambda)$, set

$$\cM^n(\Gamma)=\bigoplus_{\mu\in\Gamma}\, L_\mu \otimes S^nV \subset \cM^n.$$

Recall that we assume that
$\pi(\Lambda)$ does not lie in a complex line (cf. Assumption iii)).

\begin{lemma}\label{lemma_M^n}
Assume that condition $(\cC)$ holds.
For any $n\geq 2$, the $W_\pi$-module $\cM^n(\Gamma)$ is irreducible.

Moreover, $W_\pi$-modules $\cM^n(\Gamma)$ 
and $\cM^k(\Gamma')$ are isomorphic if and only if $n=k$ and 
$\Gamma = \Gamma'$.

\end{lemma}
The proof of this lemma will be given in Section \ref{sect:Proof-main-thm}.

\begin{rem}
The $W_\pi$-module $\cM^1(\Gamma)$ is reducible. More precisely, we have
\begin{enumerate}
\item there is an embedding $\cS_{-\frac{1}{2}\rho+\Gamma} \hookrightarrow \cM^1(\Gamma)$ of $W_\pi$-modules 
\[ L_{\mu-\frac{1}{2}\rho} \; \longmapsto \; L_{\mu-\rho}^{\cA} \otimes \left(\mu+\frac{1}{2}\rho\right) \qquad (\forall\, \mu \in \Gamma). \]
\item The cokernel of the above embedding is isomorphic to $\cS_{-\frac{3}{2}\rho+\Gamma}$ where the quotient map is 
\[ L_{\mu-\rho}^{\cA} \otimes v \; \longmapsto \; \left\langle \mu +\frac{1}{2}\rho, v \right\rangle L_{\mu-\frac{3}{2}\rho}\qquad (\forall\, \mu \in \Gamma). \]
\item The short exact sequence $\cS_{-\frac{1}{2}\rho+\Gamma} \hookrightarrow \cM^1(\Gamma) \twoheadrightarrow \cS_{-\frac{3}{2}\rho+\Gamma}$ does not split.
\end{enumerate}
\end{rem}

\subsection{Main result}
Now, we can state our main result of this article which was a conjecture in \cite{IM2}
(cf. Lemmas \ref{lemma_irr-tensor} and \ref{lemma_M^n}):

\begin{thm}\label{thm_main} Assume that $\pi: \Lambda \rightarrow \C^2$ satisfies the condition $(\cC)$. 
Let $M$ be a non-trivial simple $\Lambda$-graded cuspidal $W_\pi$-module. Then $M$ is isomorphic to one of the following $W_\pi$-modules:
\begin{enumerate}
\item $\cS_\Gamma$ for some $\Gamma \in \C^2/\pi(\Lambda)$ such that $-\rho, -2\rho \not\in \Gamma$, 
\item $\cS_{-\rho+\pi(\Lambda)}/\C L_{-\rho}$ and its restricted dual,
\item $\cM^n(\Gamma)$ for some $n \geq 2$ and $\Gamma \in \C^2/\pi(\Lambda)$. 
\end{enumerate} 
Moreover, any two different modules in the above list are non-isomorphic.
\end{thm} 
The authors have been informed that, in some particular cases, Olivier Mathieu has independently proved this conjecture.

\section{$\cA\cV_\pi$-modules}
In this section, we show that, for any simple cuspidal module $M$ over $W_\pi$, there exists a simple cuspidal $\cA\cV_\pi$-module $T$ such that $T \twoheadrightarrow M$.  

In what follows we identify $\Lambda$ and its image $\pi(\Lambda)$ in $\C^2$ to simplify notations. 
\subsection{Coinduced modules}\label{sect_coinduction}
Here, we recall a basic definition from Section $4$ of \cite{BF1}.

Let $M$ be a $\Lambda$-graded $W_\pi$-module. Recall that, on the $\C$-vector space
$\Hom_\C(\cA_\pi, M)$, one has the following standard $W_\pi$-action and $\cA_\pi$-action:
\begin{align*}
(X.\varphi)(a):=
&X.(\varphi(a))-\varphi(X.a), \\
(a.\varphi)(b):=
&\varphi(a\cdot b),
\end{align*}
where $\varphi \in \Hom_\C(\cA_\pi, M)$, $a,b \in \cA_\pi$ and $X \in W_\pi$. These actions define an $\cA\cV_\pi$-module structure on $\Hom_\C(\cA_\pi, M)$ and there is a natural surjective morphism of $W_\pi$-modules
$\Hom_\C(\cA_\pi, M) \twoheadrightarrow M; \varphi \, \mapsto \, \varphi(L_{-\rho}^{\cA})$ (cf. Proposition 4.3 in \cite{BF1}). However the coinduced module $\Hom_\C(\cA_\pi, M)$ is too big for our purposes.

Now, let $\widehat{M}$ be the subspace of $\Hom_\C(\cA_\pi, M)$ spanned by linear maps of the form $\psi(X,m)$ with $X \in W_\pi$ and $m \in M$ defined as
\[ \psi(X,m)(a)=(a.X).m \qquad \forall\, a \in \cA_\pi. \]
It follows that $\widehat{M}$ is a $\Lambda$-graded $\cA\cV_\pi$-module
and via the surjection
\break
 $\Hom_\C(\cA_\pi, M) \twoheadrightarrow M$, its image is $W_\pi.M$ (cf. Proposition 4.5 in \cite{BF1}). $\widehat{M}$ is called the \textbf{$\cA_\pi$-cover} of $M$.

 A homogeneous component
$\widehat{M}_\gamma$ is spanned by  
$\left\{ \psi(L_\lambda, M_{\gamma - \lambda}) \, | \, \lambda\in \Lambda \right\}$.

The finiteness of the multiplicities of $\widehat{M}$ is a non-trivial issue. In the next subsection, we show that if $M$ has bounded multiplicities, so does $\widehat{M}$.

\subsection{Differentiators}
For $\alpha, \beta, \xi \in \Lambda$ and $m \in \N$, we define a differentiator as
\[ \Omega_{\alpha, \beta; \xi}^{(m)}=\sum_{i=0}^m (-1)^{i}\binom{m}{i}L_{\alpha-i\xi}L_{\beta+i\xi} \in U(W_\pi). \]
By definition, differentiators satisfy the following relations:
\begin{equation}\label{eq_rel-diff-op}
\begin{split}
&\Omega_{\alpha, \beta; \xi}^{(m)}=(-1)^m \Omega_{\alpha-m\xi; \beta+m\xi; -\xi}^{(m)}, \\
&\Omega_{\alpha, \beta; \xi}^{(m)}=\Omega_{\alpha, \beta; \xi}^{(m-1)}-\Omega_{\alpha-\xi, \beta+\xi; \xi}^{(m-1)}.
\end{split}
\end{equation}
Here, we state an analogue of Theorem 3.3 in \cite{BF1}:
\begin{prop}\label{thm_BF} Let $\alpha, \beta, \gamma, \delta, \xi \in \Lambda$ and $m, r\geq 2$ be integers. Then the following identity holds in $U(W_\pi)$:
\begin{align*}
&\sum_{i=0}^m\sum_{j=0}^r (-1)^{i+j}\binom{m}{i}\binom{r}{j}\bigg(
[\Omega_{\alpha-i\xi, \beta-j\xi; \xi}^{(m)}, \Omega_{\gamma+i\xi, \delta+j\xi; \xi}^{(r)}]_+
\\
&\hbox{\hskip 7cm}-
[\Omega_{\alpha-i\xi, \gamma-j\xi; \xi}^{(m)}, \Omega_{\beta+i\xi, \delta+j\xi; \xi}^{(r)}]_+\bigg) \\
&=
\langle \beta-r\xi+\rho, \gamma-r\xi+\rho \rangle \bigg(
\langle \alpha+\rho, \delta+2r\xi+\rho \rangle \Omega_{\alpha+\delta+2r\xi, \beta+\gamma-2r\xi; \xi}^{(2m+2r)} \\
&\phantom{\langle \beta-r\xi+\rho, \gamma-r\xi+\rho \rangle} \\
&\hbox{\hskip 2cm}
+2(m\langle \xi, \delta+\rho \rangle+r\langle \alpha+\rho, \xi \rangle )
\Omega_{\alpha+\delta+(2r-1)\xi, \beta+\gamma-(2r-1)\xi; \xi}^{(2m+2r-1)}\bigg) \\
&- \langle \beta-\gamma, \xi \rangle \bigg[ 
\bigg( 2(m+r)\langle \alpha+\rho, \delta+2r\xi+\rho \rangle 
\phantom{\Omega_{\alpha+\delta+(2r-1)\xi, \beta+\gamma-(2r-1)\xi}^{(2m+2r-1)}}
 \\
& \phantom{\langle \beta-\gamma, \xi } \phantom{\int}
\hbox{\hskip 0.5cm}
-\left( m\langle \xi, \delta+\rho \rangle+r\langle \alpha+\rho, \xi \rangle \right)
\bigg) \Omega_{\alpha+\delta+(2r-1)\xi, \beta+\gamma-(2r-1)\xi; \xi}^{(2m+2r-1)} \\
&
+\left( m\langle \xi, \delta+\rho \rangle+r\langle \alpha+\rho, \xi \rangle \right)(2m+2r-1)
)\Omega_{\alpha+\delta+(2r-2)\xi, \beta+\gamma-(2r-2)\xi; \xi}^{(2m+2r-2)} 
\bigg],
\end{align*}
where $[\cdot, \cdot ]_+$ denotes the anti-commutator.
\end{prop}
\begin{proof} We follow the proof of Theorem 3.3 in \cite{BF1}.
By definition, the left hand side of the proposition reads
\begin{align*}
&\sum_{i,a=0}^m \sum_{b,j=0}^r (-1)^{i+j+a+b}\binom{m}{a}\binom{r}{b}\binom{m}{i}\binom{r}{j} \\
&\phantom{\sum_{i,a=0}^m \sum_{b, j=0}}\times 
\left(
[L_{\alpha-(i+a)\xi}L_{\beta-(j-a)\xi}, L_{\gamma+(i-b)\xi}L_{\delta+(j+b)\xi}]_+ \right.\\
&\phantom{\sum_{i,a=0}^m \sum_{b, j=0}\;}
-\left.
 [L_{\alpha-(i+a)\xi}L_{\gamma-(j-a)\xi}, L_{\beta+(i-b)\xi}L_{\delta+(j+b)\xi}]_+
\right).
\end{align*}
Switching $i$ with $a$ and $j$ with $b$ in the second term, we obtain
\begin{align*}
&\sum_{i,a=0}^m \sum_{b,j=0}^r (-1)^{i+j+a+b}\binom{m}{a}\binom{r}{b}\binom{m}{i}\binom{r}{j} \\
&\phantom{\sum_{i,a=0}^m \sum_{b, j=0}}\times 
\left(
[L_{\alpha-(i+a)\xi}L_{\beta-(j-a)\xi}, L_{\gamma+(i-b)\xi}L_{\delta+(j+b)\xi}]_+ \right.\\
&\phantom{\sum_{i,a=0}^m \sum_{b, j=0}\;}
-\left.
 [L_{\alpha-(i+a)\xi}L_{\gamma-(b-i)\xi}, L_{\beta+(a-j)\xi}L_{\delta+(j+b)\xi}]_+
\right).
\end{align*}
Setting $A=L_{\alpha-(i+a)\xi}, B=L_{\beta-(j-a)\xi}, C=L_{\gamma+(i-b)\xi}$ and $ D=L_{\delta+(j+b)\xi}$, this becomes
\[ \sum_{i,a=0}^m \sum_{b,j=0}^r (-1)^{i+j+a+b}\binom{m}{a}\binom{r}{b}\binom{m}{i}\binom{r}{j}([AB, CD]_+-[AC, BD]_+). \]
By the identity
\begin{align*}
& [AB, CD]_+-[AC, BD]_+ \\
=
& A[B,C]D+DA[C,B]+[C,D]AB+D[C,A]B+D[A,B]C+[D,B]AC,
\end{align*}
we rewrite the latter sum. In particular, as the monomials in last four terms depend only on three of the four running variables $i,j,a$ and $b$, thanks to the identities
\[ \sum_{k=0}^p (-1)^k \binom{p}{k}=\sum_{k=0}^p (-1)^k k\binom{p}{k}=0 
\]
for any integers $p \geq 2$, we see that this latter sum becomes

\begin{align*}
&\sum_{i,a=0}^m \sum_{b,j=0}^r (-1)^{i+j+a+b}\binom{m}{a}\binom{r}{b}\binom{m}{i}\binom{r}{j} (A[B,C]D+DA[C,B]) \\
=
&\sum_{i,a=0}^m \sum_{b,j=0}^r (-1)^{i+j+a+b}\binom{m}{a}\binom{r}{b}\binom{m}{i}\binom{r}{j}([A,D][B,C]+A[[B,C],D]),
\end{align*}
thanks to the identity
\[ A[B,C]D+DA[C,B]=[A,D][B,C]+A[[B,C],D].
\]

The second term becomes zero since the monomial part does not depend on the running variables $j$ and $b$. Hence, this sum becomes
\begin{align*}
& 
\sum_{i,a=0}^m \sum_{b,j=0}^r (-1)^{i+j+a+b}\binom{m}{a}\binom{r}{b}\binom{m}{i}\binom{r}{j} \\
& 
\phantom{\sum_{i,a=0}^m \sum_{b, j=0}}\times 
[L_{\alpha-(i+a)\xi},L_{\delta+(j+b)\xi}][L_{\beta-(j-a)\xi}, L_{\gamma+(i-b)\xi}] \\
=
&\sum_{i,a=0}^m \sum_{b,j=0}^r (-1)^{i+j+a+b}\binom{m}{a}\binom{r}{b}\binom{m}{i}\binom{r}{j} \\
&\phantom{\sum_{i,a=0}^m \sum_{b, j=0}}\times 
(\langle \beta+\rho, \gamma+\rho \rangle-(j-a)\langle \xi, \gamma+\rho \rangle+(i-b)\langle \beta+\rho, \xi \rangle)\\
&\phantom{\sum_{i,a=0}^m \sum_{b, j=0}}\times 
(\langle \alpha+\rho, \delta+\rho \rangle-(i+a)\langle \xi, \delta+\rho \rangle+(j+b)\langle \alpha+\rho, \xi \rangle)\\
&\phantom{\sum_{i,a=0}^m \sum_{b, j=0}}\times 
L_{\alpha+\delta-(i+a)\xi+(j+b)\xi}L_{\beta+\gamma+(i+a)\xi-(j+b)\xi}.
\end{align*}
As we have
\begin{align*}
&\langle \beta+\rho, \gamma+\rho \rangle-(j-a)\langle \xi, \gamma+\rho \rangle+(i-b)\langle \beta+\rho, \xi \rangle \\
=
&\langle \beta+\rho, \gamma+\rho \rangle+(i\langle \beta+\rho, \xi \rangle+a\langle \xi, \gamma+\rho \rangle)-(j\langle \beta+\rho, \xi \rangle+b\langle \xi, \gamma+\rho \rangle)
\\
=
&\langle \beta+\rho, \gamma+\rho \rangle\\
&+\frac{1}{2}((i+a)-(j+b))\langle \beta-\gamma, \xi \rangle -\frac{1}{2}((i-a)+(j-b))\langle \xi, \beta+\gamma+2\rho\rangle,
\end{align*}
the last term, which is skew-symmetric with respect to the substitution 
 $i\leftrightarrow a$ and $j\leftrightarrow b$, will disappear when we take the sum, leaving
\begin{align*}
&\sum_{i,a=0}^m \sum_{b,j=0}^r (-1)^{i+j+a+b}\binom{m}{a}\binom{r}{b}\binom{m}{i}\binom{r}{j} \\
&\phantom{\sum_{i,a=0}^m \sum_{b, j=0}}\times 
\left(\langle \beta+\rho, \gamma+\rho \rangle+\frac{1}{2}((i+a)-(j+b))\langle \beta-\gamma, \xi \rangle\right)\\
&\phantom{\sum_{i,a=0}^m \sum_{b, j=0}}\times 
(\langle \alpha+\rho, \delta+\rho \rangle-(i+a)\langle \xi, \delta+\rho \rangle+(j+b)\langle \alpha+\rho, \xi \rangle)\\
&\phantom{\sum_{i,a=0}^m \sum_{b, j=0}}\times 
L_{\alpha+\delta-(i+a)\xi+(j+b)\xi}L_{\beta+\gamma+(i+a)\xi-(j+b)\xi}.
\end{align*}
Making the change of variables $j \mapsto r-j$ and $b \mapsto r-b$, we obtain
\begin{align*}
&\sum_{i,a=0}^m \sum_{b,j=0}^r (-1)^{i+j+a+b}\binom{m}{a}\binom{r}{b}\binom{m}{i}\binom{r}{j} \\
&\phantom{\sum_{i,a=0}^m \sum_{b, j=0}}\times 
\left(\langle \beta-r\xi+\rho, \gamma-r\xi+\rho \rangle+\frac{1}{2}(i+j+a+b)\langle \beta-\gamma, \xi \rangle\right)\\
&\phantom{\sum_{i,a=0}^m \sum_{b, j=0}}\times 
(\langle \alpha+\rho, \delta+2r\xi+\rho \rangle-(i+a)\langle \xi, \delta+\rho \rangle-(j+b)\langle \alpha+\rho, \xi \rangle)\\
&\phantom{\sum_{i,a=0}^m \sum_{b, j=0}}\times 
L_{\alpha+\delta+2r\xi-(i+j+a+b)\xi}L_{\beta+\gamma-2r\xi+(i+j+a+b)\xi}. \\
\end{align*}
Now, set $u=i+j+a+b$. The coefficient of $L_{\alpha+\delta+(2r-u)\xi}L_{\beta+\gamma-(2r-u)\xi}$ in the above formula is the same as the coefficient of $t^{u}$ in the polynomial
\begin{align*}
&\langle \beta-r\xi+\rho, \gamma-r\xi+\rho \rangle \langle \alpha+\rho, \delta+2r\xi+\rho \rangle (1-t)^{2m+2r} \\
& + \langle \beta-\gamma, \xi \rangle \langle \alpha+\rho, \delta+2r\xi+\rho \rangle
t\frac{d}{dt}(1-t)^{2m+2r} \\
& -
\langle \beta-r\xi+\rho, \gamma-r\xi+\rho \rangle \langle \xi, \delta+\rho \rangle 
 (1-t)^{2r}t\frac{d}{dt}(1-t)^{2m} \\
& -
\langle \beta-r\xi+\rho, \gamma-r\xi+\rho \rangle \langle \alpha+\rho, \xi \rangle  (1-t)^{2m}t\frac{d}{dt}(1-t)^{2r} \\
& -
\frac{1}{2}\langle \beta-\gamma, \xi \rangle 
\left(
\langle \xi, \delta+\rho \rangle t\frac{d}{dt}\left( (1-t)^{2r}t\frac{d}{dt}(1-t)^{2m}\right)
\right. \\
&\phantom{\frac{1}{2}\langle \beta-\gamma, \xi \rangle } \left.
+\langle \alpha+\rho, \xi \rangle t\frac{d}{dt}\left( (1-t)^{2m}t\frac{d}{dt}(1-t)^{2r}\right)
\right).
\end{align*}
This polynomial can be rewritten as follows:
\begin{align*}
&\langle \beta-r\xi+\rho, \gamma-r\xi+\rho \rangle \langle \alpha+\rho, \delta+2r\xi+\rho \rangle (1-t)^{2m+2r} \\
-
&2(m+r)\langle \beta-\gamma, \xi \rangle \langle \alpha+\rho, \delta+2r\xi+\rho \rangle
t(1-t)^{2m+2r-1} \\
+
&\left( m\langle \xi, \delta+\rho \rangle+r\langle \alpha+\rho, \xi \rangle \right) \\
\times
&\bigg( 2\langle \beta-r\xi+\rho, \gamma-r\xi+\rho \rangle t(1-t)^{2m+2r-1} \\
&+ \langle \beta-\gamma, \xi \rangle t\frac{d}{dt}\left(t(1-t)^{2m+2r-1}\right)\bigg).
\end{align*}
Thus, its coefficient in $t^{u}$ is
\begin{align*}
&\langle \beta-r\xi+\rho, \gamma-r\xi+\rho \rangle \langle \alpha+\rho, \delta+2r\xi+\rho \rangle (-1)^{u}\binom{2m+2r}{u} \\
-
&2(m+r)\langle \beta-\gamma, \xi \rangle \langle \alpha+\rho, \delta+2r\xi+\rho \rangle
(-1)^{u-1}\binom{2m+2r-1}{u-1} \\
+
&\left( m\langle \xi, \delta+\rho \rangle+r\langle \alpha+\rho, \xi \rangle \right) \\
&\times \bigg[
2\langle \beta-r\xi+\rho, \gamma-r\xi+\rho \rangle (-1)^{u-1}\binom{2m+2r-1}{u-1} 
\\
&+\langle \beta-\gamma, \xi \rangle \bigg(
(-1)^{u-1}\binom{2m+2r-1}{u-1} \\
&\hbox{\hskip 4cm}
-(2m+2r-1)(-1)^{u-2}\binom{2m+2r-2}{u-2} 
\bigg)\bigg],
\end{align*}
and we obtain
\begin{align*}
&\sum_{i=0}^m\sum_{j=0}^r (-1)^{i+j}\binom{m}{i}\binom{r}{j}\bigg(
[\Omega_{\alpha-i\xi, \beta-j\xi; \xi}^{(m)}, \Omega_{\gamma+i\xi, \delta+j\xi; \xi}^{(r)}]_+\\
&\hbox{\hskip 6cm}
-
[\Omega_{\alpha-i\xi, \gamma-j\xi; \xi}^{(m)}, \Omega_{\beta+i\xi, \delta+j\xi; \xi}^{(r)}]_+\bigg) \\
=
&\langle \beta-r\xi+\rho, \gamma-r\xi+\rho \rangle \langle \alpha+\rho, \delta+2r\xi+\rho \rangle \Omega_{\alpha+\delta+2r\xi, \beta+\gamma-2r\xi; \xi}^{(2m+2r)} \\
-
&2(m+r)\langle \beta-\gamma, \xi \rangle \langle \alpha+\rho, \delta+2r\xi+\rho \rangle
\Omega_{\alpha+\delta+(2r-1)\xi, \beta+\gamma-(2r-1)\xi; \xi}^{(2m+2r-1)} \\
+
&\left( m\langle \xi, \delta+\rho \rangle+r\langle \alpha+\rho, \xi \rangle \right) \\
&\times \bigg[
2\langle \beta-r\xi+\rho, \gamma-r\xi+\rho \rangle \Omega_{\alpha+\delta+(2r-1)\xi, \beta+\gamma-(2r-1)\xi; \xi}^{(2m+2r-1)} 
\\
& \hbox{\hskip 1cm}
+\langle \beta-\gamma, \xi \rangle \bigg(
\Omega_{\alpha+\delta+(2r-1)\xi, \beta+\gamma-(2r-1)\xi; \xi}^{(2m+2r-1)}
\\
& \hbox{\hskip 4cm}
-(2m+2r-1)\Omega_{\alpha+\delta+(2r-2)\xi, \beta+\gamma-(2r-2)\xi; \xi}^{(2m+2r-2)} 
\bigg)\bigg] \\
\end{align*}
\begin{align*}
=
&\langle \beta-r\xi+\rho, \gamma-r\xi+\rho \rangle \left(
\langle \alpha+\rho, \delta+2r\xi+\rho \rangle \Omega_{\alpha+\delta+2r\xi, \beta+\gamma-2r\xi; \xi}^{(2m+2r)} \right.\\
& \hbox{\hskip 2cm}
 \left.
+2(m\langle \xi, \delta+\rho \rangle+r\langle \alpha+\rho, \xi \rangle )
\Omega_{\alpha+\delta+(2r-1)\xi, \beta+\gamma-(2r-1)\xi; \xi}^{(2m+2r-1)}\right) \\
-
&\langle \beta-\gamma, \xi \rangle \left[ 
\left( 2(m+r)\langle \alpha+\rho, \delta+2r\xi+\rho \rangle 
\phantom{\Omega_{\alpha+\delta+(2r-1)\xi, \beta+\gamma-(2r-1)\xi}^{(2m+2r-1)}}
\right. \right. \\
& \phantom{\langle \beta-\gamma, \xi }\left. \phantom{\int}
-\left( m\langle \xi, \delta+\rho \rangle+r\langle \alpha+\rho, \xi \rangle \right)
\right) \Omega_{\alpha+\delta+(2r-1)\xi, \beta+\gamma-(2r-1)\xi; \xi}^{(2m+2r-1)} \\
+&\left( m\langle \xi, \delta+\rho \rangle+r\langle \alpha+\rho, \xi \rangle \right)(2m+2r-1)
)\Omega_{\alpha+\delta+(2r-2)\xi, \beta+\gamma-(2r-2)\xi; \xi}^{(2m+2r-2)} 
\bigg].
\end{align*}
\end{proof}
A simple case is given as follows:
\begin{cor} \label{cor_diff-simple-case}
In particular, for $\beta-\gamma \in \C \xi$, we have
\begin{align*}
&\sum_{i=0}^m\sum_{j=0}^r (-1)^{i+j}\binom{m}{i}\binom{r}{j}\bigg(
[\Omega_{\alpha-i\xi, \beta-j\xi; \xi}^{(m)}, \Omega_{\gamma+i\xi, \delta+j\xi; \xi}^{(r)}]_+ \\
&  \hbox{\hskip 7cm}
-
[\Omega_{\alpha-i\xi, \gamma-j\xi; \xi}^{(m)}, \Omega_{\beta+i\xi, \delta+j\xi; \xi}^{(r)}]_+\bigg) \\
=
&\langle \beta+\rho, \gamma+\rho \rangle \left(
\langle \alpha+\rho, \delta+2r\xi+\rho \rangle \Omega_{\alpha+\delta+2r\xi, \beta+\gamma-2r\xi; \xi}^{(2m+2r)} \right.\\
&\phantom{\langle \beta+\rho, \gamma+\rho \rangle} \left.
+2(m\langle \xi, \delta+\rho \rangle+r\langle \alpha+\rho, \xi \rangle )
\Omega_{\alpha+\delta+(2r-1)\xi, \beta+\gamma-(2r-1)\xi; \xi}^{(2m+2r-1)}\right). 
\end{align*}
\end{cor}

Fix $\xi \in \Lambda \setminus \C \rho$. Recall that the Lie subalgebra $\bigoplus_{m \in \Z} \C L_{m\xi}$ is isomorphic to the Witt algebra $\bW$. 

By Corollary 3.4 of \cite{BF1}, it is known that, for any cuspidal $\bW$-module $V$ whose multiplicity is bounded by a constant, say $d$, there is $m \in \Z_{>0}$ (determined by $d$) such that the differentiator $\Omega_{r\xi, s\xi; \xi}^{(m)}$ acts trivially on $V$, for any $r, s \in \Z$. With the aid of this fact, we show the next proposition:
\begin{prop}\label{prop_annihilator-cuspidal-mod} Assume that $\pi: \Lambda \rightarrow \C^2$ satisfies the condition $(\cC)$.  
Then there exists $n \in \Z_{>0}$ such that for any $\delta, \xi \in \Lambda$ the differentiator $\Omega_{\delta, 0; \xi}^{(n)}$ acts trivially on a cuspidal $W_\pi$-module $M$. 
\end{prop}
\begin{proof} Let $\alpha, \beta, \gamma \in \Z \xi$ such that $\beta \neq \gamma$ and $\delta \in \Lambda \setminus \Z \xi$. 

As we have discussed above, the differentiator $\Omega_{r\xi, s\xi; \xi}^{(m)}$ acts trivially on $M$. Hence, the left hand side of the formula in Corollary \ref{cor_diff-simple-case} becomes $0$ and we obtain
\begin{align*}
&
\langle \alpha+\rho, \delta+2r\xi+\rho \rangle \Omega_{\alpha+\delta+2r\xi, \beta+\gamma-2r\xi; \xi}^{(2m+2r)} \\
+
&
2(m\langle \xi, \delta+\rho \rangle+r\langle \alpha+\rho, \xi \rangle )
\Omega_{\alpha+\delta+(2r-1)\xi, \beta+\gamma-(2r-1)\xi; \xi}^{(2m+2r-1)} 
= 0
\end{align*}
on $M$. Set
\[ K=\langle \alpha+\rho, \delta+2r\xi+\rho \rangle, \qquad 
L=2(m\langle \xi, \delta+\rho \rangle+r\langle \alpha+\rho, \xi \rangle ).
\]
The left hand side of the above equation can be rewritten with the aide of \eqref{eq_rel-diff-op} as
\begin{align*}
&K\Omega_{\alpha+\delta+2r\xi, \beta+\gamma-2r\xi; \xi}^{(2m+2r)}+
L\Omega_{\alpha+\delta+(2r-1)\xi, \beta+\gamma-(2r-1)\xi; \xi}^{(2m+2r-1)} \\
=
&K\Omega_{\alpha+\delta+2r\xi, \beta+\gamma-2r\xi; \xi}^{(2m+2r-1)}+
(L-K)\Omega_{\alpha+\delta+(2r-1)\xi, \beta+\gamma-(2r-1)\xi; \xi}^{(2m+2r-1)} .
\end{align*}
Writing the same formula for $\alpha'=\alpha+2(r-m)\xi, \beta'=\beta-2(r-m)\xi$ and $\xi'=-\xi$, the coefficients $K, L$ now become $K'$ and $L'$ that can be expressed as
\[ K'=K-L+2r\langle \xi, \alpha+\delta+2\rho \rangle, \qquad L'=-L. \]
Hence by \eqref{eq_rel-diff-op}, the system of linear equations we obtain is
\begin{align*}
&K\Omega_{\alpha+\delta+2r\xi, \beta+\gamma-2r\xi; \xi}^{(2m+2r-1)}+
(L-K)\Omega_{\alpha+\delta+(2r-1)\xi, \beta+\gamma-(2r-1)\xi; \xi}^{(2m+2r-1)}=0, \\
&(L'-K')\Omega_{\alpha+\delta+2r\xi, \beta+\gamma-2r\xi; \xi}^{(2m+2r-1)}+
K'\Omega_{\alpha+\delta+(2r-1)\xi, \beta+\gamma-(2r-1)\xi; \xi}^{(2m+2r-1)}=0,
\end{align*}
whence the determinant of the matrix of the coefficients is
\begin{align*}
&K\cdot K'-(L-K)\cdot (L'-K') 
=L\cdot 2r\langle \xi, \alpha+\delta+2\rho \rangle \\
=
&4r(m\langle \xi, \delta+\rho \rangle+r\langle \rho, \xi \rangle )
\langle \xi, \delta+2\rho \rangle.
\end{align*}
In particular, for $r=m$, the right hand side becomes
$4m^2\langle \xi, \delta\rangle \langle \xi, \delta+2\rho \rangle$ which is non-zero by assumption ($\cC$). Thus, 
\[ \Omega_{\alpha+\delta+2m\xi, \beta+\gamma-2m\xi; \xi}^{(4m)}=
\Omega_{\alpha+\delta+(2m-1)\xi, \beta+\gamma-(2m-1)\xi; \xi}^{(4m-1)}=0
\]
on $M$. Thus, for an appropriate choice of $\alpha, \beta$ and $\gamma$, the above formula shows that for $n = 4m$ we get $\Omega_{\delta, 0;\xi}^{(n)}=0$ on $M$ . The case $\delta \in \Z \xi$ follows from Corollary 3.4 of \cite{BF1}. 
\end{proof}

As a corollary, one can show
\begin{thm}  Assume that $\pi: \Lambda \rightarrow \C^2$ satisfies the condition $(\cC)$. Let $\widehat{M}$ be the cover of a cuspidal $W_\pi$-module $M$. 
Then the dimensions of the homogeneous components of $\widehat{M}$ are uniformly bounded. 
\end{thm}
\begin{proof}
Assume $M = \mathop\oplus_{\mu \in \Gamma} M_\mu$ for some coset $\Gamma \in \C^2 / \pi(\Lambda)$. Fix $\beta \in \Gamma$ such that $\langle \rho, \beta \rangle = 0$. By condition ($\cC$) there exists at most one such element in $\Gamma$. If $\Gamma$ contains no elements with this property, we fix $\beta$ to be an arbitrary element of $\Gamma$.

 Fix a $\Z$-basis $\vep_1, \ldots, \vep_N$ of $\Lambda$. 
Let $n$ be a natural number from Proposition \ref{prop_annihilator-cuspidal-mod}.
For any $\gamma \in \Lambda$, it is sufficient to show that $\widehat{M}_{\beta+\gamma}$ is spanned by 
\[ \bigcup_{i=1}^N\bigcup_{0\leq k_i <n } 
 \{ \psi(L_{\gamma - \sum_{i=1}^N k_i \vep_i}, x) \vert \; x \in M_{\beta + \sum_{i=1}^N k_i \vep_i}\}, 
\]
that is, we show that for any $\alpha \in \Lambda$ and $x \in M_\alpha$, $\psi(L_{\gamma-\alpha}, x)$ is expressible as a linear combination of the elements of the above form. For $\alpha=\sum_{i=1}^N k_i \vep_i$ with $0\leq k_i<n$ ($\forall\, i$), there is nothing to do. Otherwise, let $y \in M_{\beta + \alpha}$ such that $x=L_0.y$. By Proposition \ref{prop_annihilator-cuspidal-mod},
 $\Omega_{\gamma', 0; \pm\vep_k}^{(n)}$ acts trivially on $M$ for any $1\leq k \leq N$ and $\gamma' \in \Lambda$. This implies the next equality holds:
\begin{equation}\label{eq-red-index}
 \psi(L_{\gamma-\alpha}, x)=-\sum_{i=1}^n(-1)^{i}\binom{n}{i}\psi(L_{\gamma-\alpha \mp i\vep_k}, L_{\pm i\vep_k}.y). 
 \end{equation}
Indeed, one can show this identity by evaluating 
\[ \sum_{i=0}^n(-1)^{i}\binom{n}{i}\psi(L_{\gamma-\alpha\mp i\vep_k},L_{\pm i\vep_k}.y)
\]
on $L_{\delta-\rho}$ ($\delta \in \Lambda$):
\begin{align*}
&\sum_{i=0}^n(-1)^{i}\binom{n}{i}\psi(L_{\gamma-\alpha\mp i\vep_k},L_{\pm i\vep_k}.y)(L_{\delta-\rho}) \\
=
&\sum_{i=0}^n(-1)^{i}\binom{n}{i}(L_{\delta-\rho}^{\cA}.L_{\gamma-\alpha\mp i\vep_k})L_{\pm i\vep_k}.y \\
=
&\sum_{i=0}^n(-1)^{i}\binom{n}{i}L_{\gamma+\delta-\alpha\mp i\vep_k}L_{\pm i\vep_k}.y
=\Omega_{\gamma+\delta-\alpha, 0; \pm \vep_k}^{(n)}.y=0.
\end{align*}
Thus, applying appropriately \eqref{eq-red-index}, we have proved the proposition.
\end{proof}
\begin{rem} In fact, the above proof shows that the multiplicities of $\widehat{M}$ is uniformly bounded by $d \cdot n^N.$
\end{rem}

\subsection{From $\cA\cV_\pi$-modules to $W_\pi$-modules}

Let $M$ be a simple cuspidal $W_\pi$-module and $\widehat{M}$ be its $\cA_\pi$-cover. There is a surjection $\pr: \widehat{M} \twoheadrightarrow M$ (see \S \ref{sect_coinduction}). 
Let $\widehat{M}=T_0 \supsetneq T_1 \supsetneq \cdots  \supsetneq T_k=0$ be the composition series of $\cA\cV_\pi$-modules whose quotient modules $T_s/T_{s+1}$ are simple $\cA\cV_\pi$-modules. As $M$ is simple, there exists $0\leq s<k$ such that $\pr(T_0)=\cdots =\pr(T_s)=M$ and $\pr(T_{s+1})=\cdots =\pr(T_k)=0.$ In particular, we see that $\pr(T_s/T_{s+1})=M$. This means that any simple cuspidal $W_\pi$-module is a quotient of a simple cuspidal $\cA\cV_\pi$-module. In the next section, we classify simple cuspidal $\cA\cV_\pi$-modules. As an application, we obtain the classification of simple cuspidal $W_\pi$-modules.
 
\section{Classification of simple $\cA\cV_\pi$-modules of $\cA_\pi$-finite rank}

\subsection{Operator $D(\lambda)$}

Let $T$ be an $\cA\cV_\pi$-module with the finite-\ {\kern -6pt} dimensional homogeneous components, supported on a coset 
$\Gamma = \beta + \pi(\Lambda)$, $\beta \in \C^2$. 
We set $U$ to be the homogeneous component $U = T_{\beta}$, and as before, we can write
$T = \cA_\pi \otimes U$. 

We introduce the following family of operators on the finite-dimensional space $U$:

\begin{defn} For $\lambda \in \pi(\Lambda)$, we set $D(\lambda)=L^\cA_{-\lambda-\rho} \circ L^\cV_\lambda$.
\end{defn}
To be precise, the first factor $L_\lambda$ acts as an element of $W_\pi$ and the second factor $L_{-\lambda-\rho}$
as an element of $\cA_\pi$. 

Let us compute the commutation relations between $D(\lambda)$ and $D(\mu)$:
\begin{lemma}\label{lemma_comm-D} 
For $\lambda, \mu \in \pi(\Lambda)$, we have
\begin{equation} \label{DD}
[D(\lambda), D(\mu)]=\langle \lambda+\rho, \mu + \rho \rangle D(\lambda+\mu)- \langle \lambda, \mu+\rho \rangle D(\lambda) - \langle \lambda+\rho, \mu \rangle D(\mu).
\end{equation}
\end{lemma}
\begin{proof}
\begin{align*}
&[D(\lambda), D(\mu)] \\
=
&L_{-\lambda-\rho}^\cA \circ \{L_\lambda, L_{-\mu-\rho}\}^\cA \circ L_\mu^\cV  - L_{-\mu-\rho}^\cA\circ \{L_\mu, L_{-\lambda-\rho}\}^\cA \circ L_\lambda^\cV  \\
&+L_{-\lambda-\rho}^\cA \circ L_{-\mu-\rho}^\cA \circ L_\lambda^\cV \circ L_\mu^\cV
  - L_{-\mu-\rho}^\cA \circ L_{-\lambda-\rho}^\cA \circ L_\mu^\cV  \circ L_\lambda^\cV \\
=
&\langle \lambda+\rho, \mu + \rho \rangle D(\lambda+\mu) - \langle \lambda, \mu+\rho \rangle D(\lambda) - \langle \lambda+\rho, \mu \rangle D(\mu) 
\end{align*}
\end{proof}

 We denote by $\DD$ the infinite-dimensional Lie algebra spanned by the elements $D(\lambda)$, $\lambda \in \Lambda$, subject to the commutation relations
(\ref{DD}).

The $\cA\cV_\pi$-module structure can be recovered from the action of $\DD$:

\begin{lemma} \label{recover}
Let $U$ be a module for the Lie algebra $\DD$. Then $\cA_\pi \otimes U$ has a structure of an $\cA\cV_\pi$-module with $\cA_\pi$ acting
by multiplication on the left tensor factor and the action of $W_\pi$ defined as follows:
$$L_\lambda^\cV (L_\mu^\cA \otimes u) = \langle \lambda + \rho, \mu + \rho \rangle L_{\mu+\lambda}^\cA \otimes u
+ L_{\mu+\lambda}^\cA \otimes D(\lambda) u.$$
\end{lemma}

The proof of this lemma is straightforward and is left as an exercise to the reader.

Let $\vep_1, \vep_2, \ldots, \vep_N$ be a $\Z$-basis of $\pi(\Lambda)$. Write $\lambda=\sum_{i=1}^N \lambda_i \vep_i$. Our goal is to show that the operator $D(\lambda)$ is an $\End (U)$-valued polynomial in $\lambda_1, \ldots, \lambda_N$.

We will need the following auxiliary Lemma:

\begin{lemma} \label{lemma_gcd}
Let $Q_i, R_i \in \C [X_1, \ldots, X_N]$, $i=1, \ldots, r$, with \break
GCD$(Q_1, \ldots, Q_r) = Q_0$. Let $S$ be the set of common zeros of $Q_1, \ldots, Q_r$ in $\Z^N$.
Suppose $D$ is a function $\Z^N \rightarrow \C$, 
satisfying $R_i = Q_i D$, for all  $i=1, \ldots, r$ (here $Q_i, R_i$ are viewed
as functions on $\Z^N$). Then there exists a polynomial $R_0 \in \C [X_1, \ldots, X_N]$ such that $Q_0 D = R_0$ on $\Z^N \backslash S$.
\end{lemma}
\begin{proof}
The statement is trivial if $r=1$. Let us prove the claim when $r = 2$.
We have $Q_1 R_2 = Q_2 R_1$ 
as functions on $\Z^N$. Since a polynomial is determined by its values on $\Z^N$, the same equality holds in $\C [X_1, \ldots, X_N]$. 
Let $Q_0 = \GCD(Q_1, Q_2)$. Then $\frac{Q_1}{Q_0} R_2 = \frac{Q_2}{Q_0} R_1$.
Since $\C [X_1, \ldots, X_N]$ is a unique factorization domain and 
$\GCD(\frac{Q_1}{Q_0}, \frac{Q_2}{Q_0}) = 1$, there exists a polynomial $R_0$
such that $R_i = \frac{Q_i}{Q_0} R_0$ for $i = 1,2$. Comparing this with
$Q_i D = \frac{Q_i}{Q_0} Q_0 D$, we conclude that $Q_0 D = R_0$ at all points of $\Z^N$ where one of $Q_1, Q_2$ does not vanish.

The case of $r \geq 3$ may be now obtained by induction. Let 
$Q^\prime = \GCD(Q_2, \ldots, Q_{r})$. By induction assumption there exists
a polynomial $R^\prime$ such that $Q^\prime D = R^\prime$ outside the common zeros of $Q_2, \ldots, Q_r$. Since $Q_0 = \GCD(Q_1, \ldots, Q_r) = \GCD(Q_1, Q^\prime)$, we may apply the above argument with $r=2$ and get a polynomial $R_0$ such that $Q_0 D = R_0$ whenever $Q_1$ or $Q^\prime$ is non-zero. Since $Q^\prime$ is non-zero outside the set of common zeros of $Q_2, \ldots, Q_r$, we
conclude that the equality $Q_0 D = R_0$ holds outside the set of the common zeros of $Q_1, \ldots, Q_r$.
\end{proof}

Now we can establish polynomiality of $D(\lambda)$.
 
\begin{lemma} \label{lemma_D-poly}
Assume that $\pi: \Lambda \rightarrow \C^2$ satisfies the condition $(\cC)$. 
The operator $D(\lambda)$ is polynomial in $\lambda_1, \ldots, \lambda_N$ 
with coefficients in $\End (U)$.
\end{lemma}
\begin{proof}
By Lemma \ref{lemma_comm-D}, we have
\[ [D(m\vep_i), D(s\vep_i)]=\langle \rho, \vep_i \rangle (
(s-m) D((m+s)\vep_i)+ m D(m\vep_i) - s D(s\vep_i)),
\]
which is exactly the case corresponding to the Witt algebra $\bW$.
Hence, by \cite{B}, it follows that the operator $D(\lambda_i \vep_i)$ is polynomial in $\lambda_i$.

Now let us prove by induction on $k$ that $D(\lambda_1 \vep_1 + \ldots 
+ \lambda_k \vep_k)$ is a polynomial in $\lambda_1, \ldots, \lambda_k$.
The base of the induction, $k=1$ is already established.

For the induction step, we assume that $D(\lambda_1 \vep_1 + \ldots 
+ \lambda_{k-1} \vep_{k-1})$ is a polynomial. Consider the commutator
\begin{align*}
[ D(\lambda_1 \vep_1 &+ \ldots + \lambda_{k-1} \vep_{k-1}), D(\lambda_k \vep_k) ] \\
&= \langle \lambda_1 \vep_1 + \ldots + \lambda_{k-1} \vep_{k-1} + \rho,
\lambda_k \vep_k + \rho \rangle D(\lambda_1 \vep_1 + \ldots + \lambda_k \vep_k)
\\
&- \langle \lambda_1 \vep_1 + \ldots + \lambda_{k-1} \vep_{k-1}, \lambda_k \vep_k+\rho \rangle D(\lambda_1 \vep_1 + \ldots + \lambda_{k-1} \vep_{k-1}) 
\\
&- \langle \lambda_1 \vep_1 + \ldots + \lambda_{k-1} \vep_{k-1} +\rho, \lambda_k \vep_k \rangle D(\lambda_k \vep_k).
\end{align*}
By the induction assumption, the left hand side, as well as the last two summands in the right hand side, are polynomials. Set 
\begin{align*}
P_1  = &\langle \lambda_1 \vep_1 + \ldots + \lambda_{k-1} \vep_{k-1} + \rho,
\lambda_k \vep_k + \rho \rangle \\
&= \lambda_1 \lambda_k \langle \vep_1, \vep_k \rangle 
+ \ldots + \lambda_{k-1} \lambda_k \langle \vep_{k-1}, \vep_k \rangle 
\\
&+ \lambda_1 \langle \vep_1, \rho \rangle + \ldots + 
\lambda_{k-1} \langle \vep_{k-1}, \rho \rangle
- \lambda_{k} \langle \vep_{k}, \rho \rangle.
\end{align*}
Then $P_1 (\lambda_1, \ldots, \lambda_k) 
D(\lambda_1 \vep_1 + \ldots + \lambda_k \vep_k)$ is a polynomial in 
$\lambda_1, \ldots, \lambda_k$.

We can perform a change of basis $\vep^\prime_{k-1} = \vep_{k-1} + \vep_k$,
$\vep^\prime_i = \vep_i$ for $i \neq k-1$; $\lambda^\prime_k = \lambda_k - \lambda_{k-1}$, $\lambda^\prime_j = \lambda_j$ for $j \neq k$, 
and apply the same argument, obtaining that
$P^\prime_1 (\lambda^\prime_1, \ldots, \lambda^\prime_k) 
D(\lambda^\prime_1 \vep^\prime_1 + \ldots + \lambda^\prime_k \vep^\prime_k)$ is a polynomial in 
$\lambda_1, \ldots, \lambda_k$.

Note that $\lambda^\prime_1 \vep^\prime_1 + \ldots + \lambda^\prime_k \vep^\prime_k
= \lambda_1 \vep_1 + \ldots + \lambda_k \vep_k$.

Set 
\begin{align*}
P_2 (\lambda_1, &\ldots, \lambda_k) = P^\prime_1 (\lambda^\prime_1, \ldots, \lambda^\prime_k) \\
&= \lambda_1 (\lambda_k - \lambda_{k-1}) \langle \vep_1, \vep_k \rangle 
+ \ldots + \lambda_{k-1} (\lambda_k - \lambda_{k-1}) \langle \vep_{k-1} + \vep_k, \vep_k \rangle \\
&+ \lambda_1 \langle \vep_1, \rho \rangle + \ldots + 
\lambda_{k-1} \langle \vep_{k-1} + \vep_k, \rho \rangle
- (\lambda_k - \lambda_{k-1}) \langle \vep_{k}, \rho \rangle.
\end{align*}
Let us consider the difference
\begin{equation*}
P_3 = P_1 - P_2 = 
\lambda_1 \lambda_{k-1} \langle \vep_1, \vep_k \rangle + \ldots
+ \lambda_{k-1} \lambda_{k-1} \langle \vep_{k-1}, \vep_k \rangle
- 2 \lambda_{k-1} \langle \vep_{k}, \rho \rangle,
\end{equation*}
and factor out $\lambda_{k-1}$:
\begin{align*}
P_4 = \lambda_1 \langle \vep_1, \vep_k \rangle + \ldots
+  \lambda_{k-1} \langle \vep_{k-1}, \vep_k \rangle
- 2  \langle \vep_{k}, \rho \rangle \\
= \langle \lambda_1 \vep_1 + \ldots + \lambda_k \vep_k + 2 \rho, \vep_k \rangle .
\end{align*}
Finally, set
\begin{equation*}
P_5 = P_1 - \lambda_k P_4
= \lambda_1 \langle \vep_1, \rho \rangle + \ldots 
+ \lambda_k \langle \vep_k, \rho \rangle .
\end{equation*}

Then for each $j = 1,\ldots,5$, $P_j \lambda_{k-1} D(\lambda_1 \vep_1 + \ldots + \lambda_k \vep_k)$ is a polynomial in 
$\lambda_1, \ldots, \lambda_k$.

Since $\GCD(P_4, P_5) = 1$, we obtain by applying Lemma \ref{lemma_gcd} that 
$\lambda_{k-1} D(\lambda_1 \vep_1 + \ldots + \lambda_k \vep_k)$ is a polynomial in 
$\lambda_1, \ldots, \lambda_k$ outside the set of common zeros of $P_4$ and $P_5$. However by assumption, 
\begin{equation*}
\langle \delta + 2 \rho, \vep_k \rangle \neq 0
\end{equation*}
for all $\delta \in \Lambda$. Hence $P_4$ has no zeros in $\Z^k$.

By symmetry, we obtain that $\lambda_j D(\lambda_1 \vep_1 + \ldots + \lambda_k \vep_k)$ is a polynomial for all $j = 1, \ldots, k$.
Since $\GCD(\lambda_1, \ldots, \lambda_k) = 1$ in $\C[\lambda_1, \ldots, \lambda_k]$, applying again Lemma \ref{lemma_gcd}, we establish that there exists an $\End(U)$-valued polynomial $T$ such that $D = T$ on $\Z^k \backslash \{ 0 \}$.

However we know from \cite{B} that $D(\lambda_1 \vep_1)$ is given by a polynomial in $\lambda_1$ for all $\lambda_1 \in \Z$ and must agree with 
$T(\lambda_1, 0, \ldots, 0)$ for $\lambda_1 \neq 0$. Two polynomials in $\lambda_1$ that have equal values on $\Z \backslash \{ 0 \}$ must be equal, hence 
$D(0)$ is the constant term of $T$, and $D = T$ everywhere on $\Z^k$. 
This completes the step of induction.
\end{proof}

\subsection{Lie algebra $\cP$}

By Lemma \ref{lemma_D-poly}, we may assume that the operator 
$D(\lambda)$ can be expressed in the form 
\begin{equation} \label{D-expansion}
D(\lambda)=\sum_{K \in \Z_+^N} \frac{\lambda^K}{K!}P_K, 
\end{equation}
where $P_K=0$ except for finitely many $K$.
Here, $\Z_+=\Z_{\geq 0}$.  We denote by $|K|$ the sum of the components of a vector
$K \in \Z^N_+$.
It follows from the above that $P_0 = D(0)$.  Lemma \ref{lemma_comm-D}
implies that $P_0$ is a central element.
By Lemma \ref{lemma_comm-D}, 
\begin{align*}
&[D(\lambda), D(\mu)] 
=
\sum_{K,S \in \Z^N_+} \frac{\lambda^K\mu^S}{K!S!}[P_K, P_S] \\
=
&\big(\sum_{i \neq j} \lambda_i\mu_j \langle \vep_i, \vep_j \rangle+\sum_i (\mu_i-\lambda_i)\langle \rho, \vep_i\rangle \big)
\sum_{R \in \Z^N_+} \frac{(\lambda+\mu)^R}{R!}P_R \\
&-\big(\sum_{i \neq j} \lambda_i\mu_j\langle \vep_i, \vep_j \rangle-\sum_i \lambda_i \langle \rho, \vep_i\rangle \big)
\sum_{R \in \Z^N_+} \frac{\lambda^R}{R!}P_R \\
&
  -\big(\sum_{i \neq j} \lambda_i\mu_j\langle \vep_i, \vep_j \rangle+\sum_i \mu_i \langle \rho, \vep_i\rangle \big)\sum_{R \in \Z^N_+} \frac{\mu^R}{R!}P_R \\
=
&\big(\sum_{i \neq j} \lambda_i\mu_j\langle \vep_i, \vep_j \rangle +\sum_i (\mu_i-\lambda_i)\langle \rho, \vep_i\rangle \big)\sum_{|R_1|, |R_2| > 0} \frac{\lambda^{R_1}}{(R_1)!}\cdot \frac{\mu^{R_2}}{(R_2)!}P_{R_1+R_2} \\
+ &\sum_i \mu_i\langle \rho, \vep_i\rangle \sum_{|R|>0} \frac{\lambda^R}{R!}P_R - \sum_i \lambda_i \langle \rho, \vep_i\rangle \sum_{|R|>0} \frac{\mu^R}{R!}P_R
 - \sum_{i\neq j}\lambda_i\mu_j\langle \vep_i, \vep_j \rangle P_0 .
\end{align*}
Let $K, S \in \N^N$ such that $\vert K\vert, \vert S\vert >1$.  Comparing the coefficient of $\lambda^K\mu^S$, we obtain
\begin{equation} \label{comm_PK-PS-1}
[P_K, P_S] 
=
\sum_{i\neq j}K_iS_j\langle \vep_i, \vep_j \rangle P_{K+S-\vep_i-\vep_j}+\sum_i (S_i-K_i)\langle \rho, \vep_i\rangle P_{K+S-\vep_i}.
\end{equation}

Now, compute the case $K=\vep_i$ and $S=\vep_j$. Comparing the coefficients of $\lambda_i \mu_j$, we obtain
\begin{equation} \label{comm_PK-PS-2}
[P_{\vep_i}, P_{\vep_j}] 
=
\langle \rho, \vep_j \rangle P_{\vep_i} - \langle \rho, \vep_i\rangle P_{\vep_j}
-\langle \vep_i, \vep_j \rangle P_{0}.
\end{equation}

For $K=\vep_i$ and $\vert S\vert >1$, Comparing the coefficients of $\lambda_i \mu^S$, we obtain
\begin{equation} \label{comm_PK-PS-3}
[P_{\vep_i}, P_{S}] 
=
-\langle \rho,\vep_i\rangle P_S+\sum_j \langle \rho, \vep_j\rangle s_jP_{S+\vep_i-\vep_j}.
\end{equation}
Let $\cP$ be the infinite-dimensional Lie algebra
with a basis $\{ P_K \, | \, K \in \Z^N_+ \}$
subject to the commutation relations (\ref{comm_PK-PS-1})-(\ref{comm_PK-PS-3}). 

 Lemma \ref{lemma_D-poly} implies that every finite-dimensional representation
of Lie algebra $\DD$ yields a representation of $\cP$ on the same space with the property that $P_K$ act trivially for all but finitely many $K$. In fact, all finite-dimensional $\cP$-modules have that property.

\begin{lemma} \label{most-vanish}
For every finite-dimensional $\cP$-module $U$, all but finitely many $P_K$ act trivially.
\end{lemma}

 The proof of this fact is based on the following Lemma, which is a minor modification of Lemma 3.4 from \cite{B}:
\begin{lemma} \label{eigen}
Let $\theta: \cL \rightarrow \End (U)$ be a finite-dimensional representation of an infinite-dimensional Lie algebra $\cL$. For an element $x \in \cL$
and $a \in \C$ let
\begin{equation*}
\cL_{x} (a) = \mathop\cup\limits_{k=1}^\infty \Ker \left( \ad (x) - a \cdot \id \right)^k 
\end{equation*}
be the generalized eigenspace of $\ad (x)$. Then $\theta(\cL_{x}(a)) = 0$ for all
but finitely many $a$.
\end{lemma}

 We can use this result to prove Lemma \ref{most-vanish}. Indeed, it follows from (\ref{comm_PK-PS-3}) that 
\begin{equation} \label{gen-eig}
 P_S \in \mathop\oplus_{k=0}^\infty \cP_{P_{\vep_i}} ((s_i + k - 1) 
\langle \rho, \vep_i \rangle) .
\end{equation}
Then Lemma \ref{eigen} implies that there exists a constant $m \in \N$ such that $P_S$ will act as zero on $U$ for all $S$ with $|S| \geq m$.

\begin{thm} \label{equiv-cat}
There is an equivalence of categories of cuspidal $\cA\cV_\pi$-modules
supported on a coset $\beta + \pi(\Lambda)$
and finite-dimensional $\cP$-modules. 
\end{thm}
\begin{proof}
We have seen above that every cuspidal $\cA\cV_\pi$-module $T$
supported on a coset $\beta + \pi(\Lambda)$ yields a representation of $\cP$
on the homogeneous component $U = T_{\beta}$. It is straightforward to check that this functor is invertible, and every finite-dimensional module $U$ of $\cP$
produces a cuspidal 
$\cA\cV_\pi$-module structure on $\cA_\pi \otimes U$. Indeed, we can recover the relation (\ref{DD}) from the Lie brackets in $\cP$ and then 
by Lemma \ref{recover} we get a $W_\pi$-module structure on $\cA_\pi \otimes U$.
\end{proof}



 Our next goal is to describe irreducible finite-dimensional representations of
$\cP$. 


\subsection{Irreducible representations of $\cP$}

 The following Lemma is quite useful for the study of irreducible finite-dimensional representations: 
\begin{lemma}[cf. \cite{FH}, Proposition 9.17]\label{lemma_FH}
Let $U$ be an irreducible finite-dimensional representation of a finite-dimensional Lie algebra $\fg$. Then 
$$[\fg, \fg] \cap \mathrm{Rad}\, \fg$$
acts trivially on $U$, where $\mathrm{Rad}\, \fg$ stands for the solvable radical of $\fg$.
\end{lemma}

\begin{lemma} \label{commutant}
For every $K \in \Z_+^N$ with $|K| \geq 2$ we have $P_K \in [\cP, \cP]$.
\end{lemma}
\begin{proof}
Using \eqref{gen-eig}, we see 
that every $P_K$ with $k_i \geq 2$ is in the commutator subalgebra $[\cP, \cP]$. Let us show that for $K = \vep_{i_1} + \ldots + \vep_{i_r}$ with $r \geq 2$ 
and $i_1, \ldots, i_r$ distinct, $P_K \in [\cP, \cP]$.
Indeed, 
\begin{align*}
[P_{\vep_{i_1}}, P_{2\vep_{i_2} + \vep_{i_3} + \ldots + \vep_{i_r}}]
=& - \langle \rho, \vep_{i_1} \rangle
P_{2\vep_{i_2} + \vep_{i_3} + \ldots + \vep_{i_r}} \\
&+ 2 \langle \rho, \vep_{i_2} \rangle
P_{\vep_{i_1} + \vep_{i_2} + \ldots + \vep_{i_r}} \\
& + \langle \rho, \vep_{i_3} \rangle
P_{\vep_{i_1} + 2\vep_{i_2} + \vep_{i_4} + \ldots + \vep_{i_r}} + \ldots \\
&  + \langle \rho, \vep_{i_r} \rangle
P_{\vep_{i_1} + 2\vep_{i_2} + \vep_{i_3} + \ldots + \vep_{i_{r-1}}}.
\end{align*}
Now the left hand side, as well as all terms, except possibly the second one, in the right hand side, are in $[\cP, \cP]$. Hence, 
$P_{\vep_{i_1} + \vep_{i_2} + \ldots + \vep_{i_r}}$ is also in the commutator subalgebra.
\end{proof}

\begin{lemma} \label{PK2}
Let $U$ be a finite-dimensional irreducible $\cP$-module. Then all $P_K$ with
$|K| > 2$ act trivially on $U$.
\end{lemma}
\begin{proof}
First, by Lemma \ref{most-vanish}, every finite-dimensional representation of 
$\cP$ factors through a quotient
\begin{equation*}
\cP^\prime = \cP / \Span \{ P_S \, \big| \, |S| \geq m \}. 
\end{equation*}
for some $m \in \N$. We can immediately see from \eqref{comm_PK-PS-1} and
\eqref{comm_PK-PS-3} that
elements $P_K$ with $|K| > 2$ form a solvable ideal in $\cP^\prime$ and hence belong to the solvable radical. Combining this with Lemma \ref{commutant}, we conclude that all $P_K$ with $|K| > 2$ belong to
$[\cP^\prime, \cP^\prime] \cap \Rad \, \cP^\prime$ and hence vanish on $U$ according to Lemma \ref{lemma_FH}.  
\end{proof}

Hence, an irreducible finite-dimensional $\cP$-module $U$ induces an irreducible finite-dimensional $\overline{\cP}$-module structure on $U$, where $\overline{\cP}=\cP/ \bigoplus_{\vert K\vert>2} \C P_K$. The image of $P_K$ in $\overline{\cP}$ for $\vert K\vert \leq 2$ will be denoted by the same letter.

Let us determine the reductive quotient $\lcP / (\Rad \, \lcP \cap [\lcP, \lcP])$.

First of all, we note that $\lcP$ is a semidirect product, $\lcP = \lcP_1 \ltimes \lcP_2$, where the subalgebra $\lcP_1$ is spanned by $P_K$ with $|K| = 0, 1$, and the ideal
$\lcP_2$ is spanned by $P_K$ with $|K| = 2$.


\begin{lemma} \label{P1}
Lie algebra $\lcP_1$ is a solvable Lie algebra with an abelian commutator subalgebra.
The map $\tau : \lcP_1 \rightarrow V \cong \C^2 $, given by
\begin{align*}
\tau(P_{\vep_i}) &= \vep_i \\
\tau(P_0) &= \rho,
\end{align*}
has $[\lcP_1, \lcP_1]$ as its kernel and induces an isomorphism between
the quotient $\lcP_1 / [\lcP_1, \lcP_1]$  and a 2-dimensional abelian Lie algebra $V$.
\end{lemma}
\begin{proof}
Using \eqref{comm_PK-PS-2}, we can compute
\begin{align*}
\tau ( [P_{\vep_i}, P_{\vep_j} ]) &= 
\langle \rho, \vep_j \rangle  \vep_i
- \langle \rho, \vep_i \rangle  \vep_j 
- \langle \vep_i, \vep_j \rangle \rho = 0.
\end{align*}
Here we used the fact that
$$\langle \rho, \vep_j \rangle \vep_i - \langle \rho, \vep_i \rangle \vep_j=\langle \vep_i, \vep_j \rangle \rho.$$
This tells us that $\tau$ is a homomorphism of Lie algebras with $[\lcP_1 ,\lcP_1] \subset \Ker \tau$. In fact, it is easy to see that we actually have 
the equality
$[\lcP_1 ,\lcP_1] = \Ker \tau$. 
It is straightforward to verify that $[\lcP_1 ,\lcP_1]$ is abelian. 
It is also obvious that $\tau$ is surjective.
\end{proof}

Consider the space $\C^N = \C \otimes_\Z \Lambda$ with a skew symmetric bilinear form $\langle \cdot \vert \cdot \rangle$ induced from $\pi(\Lambda) \subset \C^2$. Then the space $S^2 \C^N$ has a Lie algebra structure coming from the Poisson bracket on $S^\bullet \C^N$.

\begin{lemma} \label{P2CN}
There is an isomorphism of Lie algebras $\eta:  \lcP_2 \rightarrow S^2 \C^N$, given by
\begin{equation*}
\eta(P_{\vep_i + \vep_j}) = \vep_i \vep_j .
\end{equation*}
\end{lemma}
\begin{proof}
This follows immediately from the Lie bracket \eqref{comm_PK-PS-1}:
\begin{align*}
[P_{\vep_a + \vep_b}, &P_{\vep_c + \vep_d} ] \\
&= \langle \vep_a \vert \vep_c \rangle P_{\vep_b + \vep_d}
+ \langle \vep_a \vert \vep_d \rangle P_{\vep_b + \vep_c}
+ \langle \vep_b \vert \vep_c \rangle P_{\vep_a + \vep_d}
+ \langle \vep_b \vert \vep_d \rangle P_{\vep_a + \vep_c}.
\end{align*}
\end{proof}

\begin{lemma} \label{P2}
There is a homomorphism $\pi_*: S^2 \C^N \rightarrow S^2 \C^2 \cong \fsl_2$, given
by
\begin{equation*}
\pi_* (\lambda \mu) = \pi(\lambda) \pi(\mu), \text{ \ for \ } 
\lambda, \mu \in \C^N.
\end{equation*}
The kernel of this homomorphism is a solvable ideal in $S^2 \C^N$.
\end{lemma}
\begin{proof}
The Lie bracket in $S^2 \C^N$ is 
\begin{equation} \label{S2CN}
[\alpha \beta, \gamma \delta] = 
\langle \alpha \vert \gamma \rangle \beta \delta
+ \langle \alpha \vert \delta \rangle \beta \gamma
+ \langle \beta \vert \gamma \rangle \alpha \delta
+ \langle \beta \vert \delta \rangle \alpha \gamma.
\end{equation}
Since by definition $\langle \alpha \vert \gamma \rangle = 
\langle \pi(\alpha), \pi(\gamma) \rangle$,
the map $\pi_*$ is a homomorphism of Lie algebras.
The kernel of $\pi_*$ is spanned by $\alpha \beta$ with $\beta \in \Ker \pi$.
Consider also the subspace $Z \subset S^2 \C^N$ spanned by 
$\alpha \beta$ with both $\alpha, \beta \in \Ker \pi$. We immediately see 
from \eqref{S2CN} that $Z$ is central in $S^2 \C^N$. Moreover, it follows from
\eqref{S2CN} that $[ \Ker \pi_*, \Ker \pi_*] \subset Z$. Hence
$\Ker \pi_*$ is a solvable ideal in $S^2 \C^N$, as claimed.
\end{proof}

Identifying $\lcP_2$ with $S^2 \C^N$, we can rewrite the Lie bracket 
\eqref{comm_PK-PS-3} in the following way:
\begin{equation} \label{P2S2}
[P_{\vep_i}, \lambda \mu] = - \langle \rho ,\pi(\vep_i) \rangle \lambda \mu
+ \langle \rho, \pi(\lambda) \rangle \vep_i  \mu
+ \langle \rho, \pi(\mu) \rangle \vep_i  \lambda .
\end{equation}

It follows immediately from this formula that the pull-back of $\Ker \pi_*$ to
$\lcP_2$ is an ideal in $\lcP$. Since this ideal is solvable and belongs to
$[\lcP, \lcP]$ by Lemma \ref{commutant}, it will vanish in every finite-dimensional irreducible representation of $\lcP$. Hence irreducible representations of $\lcP$ factor through the quotient
\begin{equation*}
\lcP_1 \ltimes S^2 \C^2 .
\end{equation*}
The action of $\lcP_1$ on $S^2 \C^2$ defines the homomorphism $\lcP_1 \rightarrow \Der (S^2 \C^2)$. Since $S^2 \C^2$ is isomorphic to $\fsl_2$, we
have $\Der (S^2 \C^2) \cong S^2 \C^2$, and we obtain a homomorphism
\begin{equation*}
\varphi: \ \lcP_1 \rightarrow S^2 \C^2.
\end{equation*}
\begin{lemma} \label{phi-map}
The map $\varphi$ is given by the formulas: 
$$\varphi(P_{\vep_i}) = 
\frac{1}{2} \pi(\vep_i)\rho, \ \ \varphi(P_0) = 0.$$
\end{lemma}
\begin{proof}
It is sufficient to calculate on each ${\lambda \mu}$ for $\lambda, \mu \in \C^2$: by \eqref{comm_PK-PS-1}, \eqref{P2S2}
\begin{align*}
&\frac{1}{2}[{\pi(\vep_i) \rho}, {\lambda \mu}]-[{P_{\vep_i}}, {\lambda \mu}] \\
=
&\frac{1}{2}(\langle \pi(\vep_i), \lambda \rangle {\rho \mu}+\langle \pi(\vep_i), \mu \rangle {\rho \lambda}+\langle \rho, \lambda \rangle {\pi(\vep_i) \mu}+\langle \rho, \mu \rangle {\pi(\vep_i) \lambda}) \\
&-(-\langle \rho, \pi(\vep_i)\rangle {\lambda \mu}+\langle \rho, \lambda\rangle {\pi(\vep_i) \mu}+\langle \rho, \mu \rangle {\pi(\vep_i) \lambda}) \\
=
&\frac{1}{2}(\langle \pi(\vep_i), \lambda \rangle {\rho \mu}+\langle \pi(\vep_i), \mu \rangle {\rho \lambda})+\langle \rho, \pi(\vep_i) \rangle {\lambda \mu} \\
&-\frac{1}{2}(\langle \rho, \lambda \rangle {\pi(\vep_i) \mu}+\langle \rho, \mu \rangle {\pi(\vep_i) \lambda}). \\
\end{align*}
Now, the identities
\[
 \langle \pi(\vep_i), \nu \rangle \rho - \langle  \rho,  \nu \rangle \pi(\vep_i) = -\langle \rho, \pi(\vep_i) \rangle \nu,
\qquad \nu \in \{\lambda, \mu \},
\]
imply that $\frac{1}{2}[{\pi(\vep_i) \rho}, {\lambda \mu}]-[P_{\vep_i}, {\lambda \mu}]=0$.
\end{proof}

 By this Lemma, the semi-direct product $\lcP_1 \ltimes S^2 \C^2$ splits into a direct sum of ideals, with the isomorphism map
$$ \lcP_1 \ltimes S^2 \C^2 \rightarrow \lcP_1 \oplus S^2 \C^2,$$
given by $(x,y) \mapsto (x, \varphi(x) + y)$, with $x \in \lcP_1$ and 
$y \in S^2 \C^2$.

Applying Lemma \ref{P1}, we see that for $x \in \Ker \tau$, elements 
$x - \varphi(x) \in \lcP$ form an abelian ideal and act trivially on every irreducible finite-dimensional $\lcP$-module. Then irreducible finite-dimensional representations of $\lcP$ are described by the irreducible finite-dimensional representations of the reductive Lie algebra $\C^2 \oplus S^2 \C^2$,
where the first direct summand is an abelian Lie algebra. 

 Fix a vector $\rho^\dagger \in \C^2$ such that $\langle \rho, \rho^\dagger \rangle = 1$.

The following theorem gives a description of finite-dimensional irreducible representations of $\cP$:

\begin{thm} \label{class-P}
Every irreducible finite-dimensional module for the Lie algebra $\cP$ is given 
by an irreducible finite-dimensional module $U$ for the Lie algebra 
$S^2 \C^2 \cong \fsl_2$, and two scalars $K_0, K_1 \in \C$. The action of $\cP$ on $U$ is the following:
\begin{align} 
P_0 &\mapsto K_0 \id_U, \label{P-action1} \\ 
P_{\vep_i} &\mapsto \langle \vep_i, K_0 \rho^\dagger + K_1 \rho \rangle \id_U
+ \frac{1}{2} \pi(\vep_i) \rho, \label{P-action2} \\ 
P_{\vep_i + \vep_j} &\mapsto \pi(\vep_i) \pi(\vep_j), \label{P-action3} \\
P_K &\mapsto 0, \text{\ for \ } |K| \geq 3. \label{P-action4}  
\end{align}
\end{thm}
\begin{proof}
By Lemmas \ref{commutant}-\ref{phi-map}, finite-dimensional irreducible representations of $\cP$ 
are pull-backs of simple modules for the reductive Lie algebra $\C^2 \oplus S^2 \C^2$ via
the map $\gamma: \cP \rightarrow \C^2 \oplus S^2 \C^2$, given by
\begin{align*}
\gamma (P_0) &= \rho \in \C^2, \\
\gamma (P_{\vep_i}) &= \pi(\vep_i) + \frac{1}{2} \pi(\vep_i) \rho, \\
\gamma (P_{\vep_i + \vep_j}) &= \pi(\vep_i) \pi(\vep_j) \in S^2 \C^2, \\
\gamma (P_K) &= 0, \text{\ for \ } |K| \geq 3.
\end{align*}
Every finite-dimensional simple module for the Lie algebra $\C^2 \oplus S^2 \C^2$ is a simple
$S^2 \C^2 \cong \fsl_2$-module $U$, on which elements of $\C^2$ act as scalars. Assuming that
$\rho \in \C^2$ acts as $K_0 \id_U$, and $\rho^\dagger \in \C^2$ acts as $-K_1 \id_U$ and 
taking into account the decomposition
$$ \pi(\vep_i) = \langle \vep_i, \rho^\dagger \rangle \rho - \langle \vep_i, \rho \rangle \rho^\dagger,$$
we obtain the claim of the theorem.
\end{proof}

Using Theorem \ref{class-P} we can prove that the class of modules introduced in section 1.3 exhausts
all simple $\Lambda$-graded $\cA\cV_\pi$-modules of $\cA_\pi$-finite rank.

\begin{thm} \label{class-AV}
Assume that $\pi: \Lambda \rightarrow \C^2$ satisfies the condition $(\cC)$.
Then every simple $\Lambda$-graded $\cA\cV_\pi$-module $T$ of $\cA_\pi$-finite rank is isomorphic to $\cM^n (\Gamma)$ for 
some $n \geq 0$ and $\Gamma \in \C^2 / \pi(\Lambda)$.
\end{thm}
\begin{proof}
Let $U$ be a homogeneous component of $T$, $U = T_{a}$ for some $a \in \C$. Then $U$ is a finite-dimensional module
for the Lie algebra $\cD$ with $D(0)$ acting by $a \id_U$. Since $T$ is a simple $\cA\cV_\pi$-module, it follows from Lemma
\ref{recover} that $U$ is an irreducible representation for $\cD$. By Lemma \ref{lemma_D-poly}, $U$ admits the action of the 
Lie algebra $\cP$.
Since the action of $\cD$ can be expressed in terms of the action of
$\cP$, $U$ is irreducible as a $\cP$-module. Using Theorem \ref{class-P}, we conclude that $U$ is an irreducible
$S^2 \C^2 \cong \fsl_2$-module with the action of $\cP$ given by \eqref{P-action1} - \eqref{P-action4}. 
Combining Lemma \ref{recover}
with \eqref{D-expansion} and \eqref{P-action1} - \eqref{P-action4}, we recover the $\cA\cV_\pi$-module structure on $T = \cA_\pi \otimes U$:
\begin{align*}
L_\lambda^\cV &(L_\mu^\cA \otimes u) 
= \langle \lambda + \rho, \mu + \rho \rangle L_{\lambda + \mu}^\cA \otimes u \\
+& L_{\lambda + \mu}^\cA \otimes K_0 u
+ \sum_{i=1}^N \lambda_i \langle \vep_i, K_0 \rho^\dagger + K_1 \rho \rangle L_{\lambda + \mu}^\cA \otimes u \\
+& \frac{1}{2}  L_{\lambda + \mu}^\cA \otimes \left(
\sum_{i=1}^N \lambda_i (\pi(\vep_i) \rho) u
+ \sum_{i \neq j} \lambda_i \lambda_j (\pi(\vep_i) \pi(\vep_j )) u
+ \sum_{i=1}^N \lambda_i^2 \pi(\vep_i)^2 u \right) \\
=& \langle \lambda + \rho, \mu + \rho \rangle L_{\lambda + \mu}^\cA \otimes u 
+ \langle \lambda + \rho,  K_0 \rho^\dagger + K_1 \rho \rangle L_{\lambda + \mu}^\cA \otimes u \\
+&  \frac{1}{2} L_{\lambda + \mu}^\cA \otimes \lambda (\lambda + \rho) u .
\end{align*}
This shows that $T$ is isomorphic to $\cM^n (\beta + \Lambda)$ with $\beta = K_0 \rho^\dagger + K_1 \rho+\pi(\Lambda)$ and $n = \dim U - 1$.
\end{proof}
\begin{cor} \label{M-irr}
Assume that $\pi: \Lambda \rightarrow \C^2$ satisfies the condition $(\cC)$.
Then $\cM^n (\Gamma)$ is irreducible as an $\cA\cV_\pi$-module for every 
$\Gamma \in \C^2 / \pi(\Lambda)$ and  $n \geq 0$.
\end{cor}
\begin{proof}
Every $\cV$-submodule of $\cM^n (\Gamma)$ is $\Lambda$-graded, and 
every $\Lambda$-graded $\cA$-submodule of $\cM^n (\Gamma)$ is of the form $\cS_\Gamma \otimes U^\prime$ for some subspace 
$U^\prime$ in $S^n \C^2$. But then $U^\prime$ is a $\cP$-submodule in $S^n \C^2$ and hence is also an $\fsl_2$-submodule.
However $S^n \C^2$ is an irreducible $\fsl_2$-module. Hence every $\cA\cV_\pi$-submodule in $\cM^n (\Gamma)$ is trivial.
\end{proof}

\section{Proof of the Main Theorem}\label{sect:Proof-main-thm}
In this section, we prove Lemma \ref{lemma_M^n} and the Main Theorem \ref{thm_main}.
We shall notice that the hypothesis $\mathrm{rank}\, \Lambda>1$ is essential. 

In this section, we identify the lattice $\Lambda$ with its image $\pi(\Lambda)$ in $\C^2$.

\subsection{Isomorphism $S^2 \C^2 \cong \fsl_2$ }
Let $\xi, \eta \in \Lambda \setminus 0 \subset \C^2$ be two linearly independent elements.
Set 
\[ e=-\frac{1}{2\langle \xi, \eta\rangle}\xi^2, \qquad f=\frac{1}{2\langle \xi, \eta\rangle} \eta^2, \qquad h=-\frac{1}{\langle \xi, \eta \rangle} \xi \eta.
\]
The triplet $\{e,f,h\}$ forms a standard $\fsl_2$-triplet, i.e., they satisfy 

\centerline{$[h,e]=2e, \, [h,f]=-2f$ and $[e,f]=h$.} 

\subsection{Irreducibility of $\cM^n(\Gamma)$ with $n\geq 2$}
Every $W_\pi$-submodule $N$ of $\cM^n(\Gamma)$ is $\Lambda$-graded, i.e., 
$N=\bigoplus_{\mu \in \Gamma} N_\mu$ where $N_\mu=N \cap \cM^n_\mu$. Fix $\mu \in \Gamma$ and let $N'_\mu$ be the subspace of $S^nV$ such that $N_\mu=\C L_{\mu-\rho}^{\cA} \otimes N'_\mu$. For any $\lambda \in \Lambda$, the element $w=L_{\lambda-\gamma}^{\cV}.(L_\gamma^{\cV}. (L_{\mu-\rho}^{\cA}\otimes n'))$ with $n' \in N'_\mu$ is an element of $N'_{\lambda+\mu}$ for any $\gamma \in \Lambda$, and $w$ is a degree 4 polynomial in $\gamma$. Varying $\gamma \in \Lambda$, we see that each coefficient of monomials in $\gamma$ is an element of $N'_{\lambda+\mu}$. In particular, the quartic component of $w$ in $\gamma$ sits in $N'_{\lambda+\mu}$ and is given by
\[ \frac{1}{4}L_{\lambda+\mu-\rho}^{\cA} \otimes \{ \gamma^2, \{\gamma^2, n'\}\}. \]
Taking $\gamma= \xi, \eta, \xi\pm \eta$, we see that
\[  \{\xi^2, \{\eta^2, n'\}\}+\{\eta^2, \{\xi^2, n'\}\} +4
\{\xi\eta, \{\xi\eta, n'\}\}
\]
are elements of $N'_{\lambda+\mu}$. In addition, these two elements corresponds to 
$-4\langle \xi, \eta \rangle^2(ef+fe-h^2)$ via the isomorphism given in the previous subsection. 
Hence $N'$ is closed under the operators $L_{\mu-\rho}^{\cA}\otimes (ef+fe-h^2)$ for all $\mu \in \Lambda$.
By the assumption $n\geq 2$, the element $ef+fe-h^2$ acts on $S^nV$ as an invertible semi-simple element. Thus $N'=N'_\mu$ does not depend on the choice of $\mu \in \Gamma$, namely, $N=\cS_\Gamma \otimes N' \subset \cM^n(\Gamma)$. 
This shows that $N$ is not just a $W_\pi$-submodule of $\cM^n(\Gamma)$, but in fact it is an $\cA \cV_\pi$-submodule. But by Corollary \ref{M-irr},  $\cM^n(\Gamma)$
is irreducible as an $\cA \cV_\pi$-module.

This completes the proof of the first half of Lemma  \ref{lemma_M^n}.


Let us sketch the proof of the second claim of Lemma \ref{lemma_M^n}. We need to show that $\cM^n(\Gamma)$ and $\cM^k(\Gamma^\prime)$ are non-isomorphic as 
$W_\pi$-modules unless $n=k$ and $\Gamma = \Gamma^\prime$. 

Comparing dimensions of homogeneous components, we see immediately that 
$\cM^n(\Gamma) \cong \cM^k(\Gamma^\prime)$ implies $n=k$. Since 
$\Gamma$ and $\Gamma^\prime$ are cosets of $\pi(\Lambda)$, it is sufficient to show that $\Gamma \cap \Gamma^\prime \neq \emptyset$.

Let us fix non-zero $\xi \in \pi(\Lambda)$ and $\mu \in \Gamma$. We may choose $\mu$
in such a way that 
$$\left| \frac{\langle \mu, \xi \rangle}{\langle \rho, \xi \rangle} \right| > \frac{n}{2} +2.$$
Suppose a homogeneous component $\cM^n(\Gamma)_\mu$  is mapped to 
$\cM^n(\Gamma^\prime)_{\mu^\prime}$ under the isomorphism. We want to show that $\mu = \mu^\prime \in \C^2$. Comparing the actions of $L_0$ on these spaces,
we see that 
\begin{equation}
\label{mumu}
\langle \mu, \rho \rangle =  \langle \mu^\prime, \rho \rangle.
\end{equation}

Next consider a $\bW$-subalgebra 
$$\mathop\oplus_{m \in \Z} \C L_{m \xi} \subset W_\pi .$$
Let us study the action of this subalgebra on the subspace 
$$V_{\mu,\xi} = \mathop\oplus_{k\in\Z} L_{\mu + k \xi} \otimes  S^n \C^2 \subset 
\cM^n(\Gamma) .$$
This subsace is a cuspidal $\bW$-module and has a composition series of tensor $\bW$-modules. Let us recall that a tensor $\bW$-module $T(\alpha, \beta)$, $\alpha, \beta \in \C$,  is a module with basis $\{ v_k \, | \, k\in\Z \}$ and the action
$$L_m v_k = (k + \alpha + m \beta) v_{k+m} .$$
It is well-known that for $\beta \neq 0, 1$ tensor modules $T(\alpha, \beta)$ are 
irreducible and $T(\alpha, \beta) \cong T(\alpha, \beta^\prime)$ implies $\beta = \beta^\prime$.

Elements $L_{m\xi}$ act in the following way:
\begin{align*}
L_{m\xi} (L_{\mu+k\xi} \otimes u) 
= &\langle m\xi +\rho, \mu + k\xi + \rho \rangle   L_{\mu+(k+m)\xi} \otimes u \\
&+ \frac{1}{2}  L_{\mu+(k+m)\xi} \otimes \left\{ m^2 \xi^2 + m \xi \rho, u \right\} .
\end{align*}
Note that $\xi^2$ is a nilpotent element of $S^2 \C^2$ and $\xi \rho$ is a semisimple
element. From this it is easy to see that $V_{\mu,\xi}$ has a unique tensor submodule 
$$\mathop\oplus_{k\in\Z} \C  L_{\mu + k \xi} \otimes \xi^n$$
with the action
\begin{align*}
&L_{m\xi} (L_{\mu+k\xi} \otimes \xi^n) = \\
& \ \ \ \left( k \langle \rho, \xi \rangle
- m  \langle \rho, \xi \rangle
+ m \langle \xi, \mu \rangle
+  \langle \rho, \mu \rangle
+ \frac{mn}{2}  \langle \rho, \xi \rangle \right)
 L_{\mu+(k+m)\xi} \otimes  \xi^n .
\end{align*}
Since the isomorphism with $\bW$ is given by $L_{m \xi} \mapsto 
 \langle \rho, \xi \rangle L_m$, we see that this is a tensor module $T(\alpha, \beta)$ with 
$$\alpha = \frac{ \langle \rho, \mu \rangle}{ \langle \rho, \xi \rangle}
\text{ \ and \ }
\beta = \frac{n}{2} +  \frac{ \langle \xi, \mu \rangle}{ \langle \rho, \xi \rangle} - 1.$$
From our assumptions on $\mu, \xi$ we see that $|\beta| > 1$. Since $\beta$ is 
uniquely determined by $V_{\mu,\xi}$, we conclude that 
$\langle \mu, \xi \rangle =  \langle \mu^\prime, \xi \rangle$, which together with (\ref{mumu}) implies that 
$\mu = \mu^\prime$ and $\Gamma = \Gamma^\prime$.

\subsection{Proof of the Main Theorem}
Let $M$ be a non-trivial simple $\Lambda$-graded cuspidal $W_\pi$-module. As we have seen in Section 2.3, there exists a simple cuspidal $\cA \cV_\pi$-module $T$
with a surjective homomorphism $T \twoheadrightarrow M$. By Theorem \ref{class-AV}, $T$ is isomorphic to $\cM^n(\Gamma)$ for some coset $\Gamma \in \C^2 / \Lambda$ and $n \geq 0$. If $n \geq 2$ then by Lemma \ref{lemma_M^n}, $T$ is irreducible as a $W_\pi$-module, and $M$ is isomorphic to $\cM^n(\Gamma)$.
If $n = 1$, every simple quotient of $\cM^1(\Gamma)$ is also a quotient of $\cM^0(-\frac{3}{2} \rho + \Gamma)$, and hence the case $n=1$ may be excluded from our classification. Finally, if $n = 0$, the module $\cM^0(\Gamma)$ is isomorphic to $\cS_\Gamma$, and the structure of these modules is described in Lemma 
\ref{lemma_irr-tensor}.


\end{document}